%% file: main.tex
\documentclass[onefignum,onetabnum]{siamonline190516}

\input{ex_shared}
\ifpdf
\hypersetup{
  pdftitle={Statistical analysis of random objects via  metric measure Laplacians},
    pdfauthor={G. Mordant, and A. Munk}
}
\fi

\usepackage{amsmath}
\usepackage{amsfonts}
\usepackage{dsfont}
\usepackage{amssymb}
\usepackage{txfonts}

\usepackage{dsfont}
\usepackage{tikz}
\usetikzlibrary{positioning,shadows,arrows}
\usepackage{pgfplots}
\usepackage{url}
\usepackage{comment}

\newcommand{\tr}{\operatorname{tr}}
\newcommand{\reals}{\mathbb{R}}
\newcommand{\ind}{\mathds{1}}
\newcommand{\diff}{\mathrm{d}}
\newcommand{\normal}{\mathcal{N}}
\newcommand{\dto}{\rightsquigarrow}

\usepackage{mathtools}

\providecommand{\trace}{\operatorname{tr}}
\providecommand{\diag}{\operatorname{diag}}
\providecommand{\normal}{\mathcal{N}}
\providecommand{\reals}{\mathbb{R}}
\providecommand{\QQ}{\mathbb{Q}}

\providecommand{\Oh}{\mathrm{O}}

\providecommand{\1}{\mathds{1}}
\providecommand{\ind}{\mathds{1}}

\providecommand{\dto}{\rightsquigarrow}
\providecommand{\dist}{\mathfrak{d}}
\providecommand{\eps}{\varepsilon} 
\providecommand{\fclass}{\mathcal{F}} 
\providecommand{\diff}{\mathrm{d}}
\providecommand{\prob}{\mathrm{P}}

\providecommand{\expec}{\mathbb{E}}
\providecommand{\Gproc}{\mathbb{G}}
\providecommand{\argmin}{\operatornamewithlimits{\arg\min}}

\providecommand{\e}{\mathbf{e}}

\usepackage{color}
\definecolor{darkraspberry}{rgb}{0.53, 0.15, 0.34}
\definecolor{britishracinggreen}{rgb}{0.0, 0.26, 0.15}
\definecolor{greenalt}{rgb}{0.0, 0.8, 0.15}
\definecolor{burntumber}{rgb}{0.54, 0.2, 0.14}
\definecolor{royalblue}{RGB}{0,78,156}
\definecolor{olive}{RGB}{128,128,0}

\begin{document}

\maketitle

\begin{abstract}
In this paper, we consider a certain convolutional Laplacian for metric measure spaces and investigate its potential for the statistical analysis of complex objects.
The spectrum of that Laplacian serves as a signature of the space under consideration and the eigenvectors provide the principal directions of the shape, its harmonics. These concepts are used to assess the similarity of objects or understand their most important features in a principled way which is illustrated in various examples.
Adopting a statistical point of view,  we define a mean spectral measure and its empirical counterpart. The corresponding limiting process of interest is derived and statistical applications are discussed.
\end{abstract}



\section{Introduction}
\label{sec: Intro}
With new, more and more complex acquired data, the need for tailored analysis tools presses. In particular, methods which respect the underlying geometric structures in the data have become important in various fields such as computer vision \cite{ezuz2017gwcnn,torralba2003context}, engineering \cite{schmiedl2015shape} or molecular biology \cite{agarwal2007fast}, to provide only a few examples.\\
\bgroup
To capture important geometric features of such data, in a fast-developing line of work \cite[and the references therein]{Memoli2007,memoli2011gromov,solomon2016entropic,chowdhury2021quantized,alaya2022theoretical}, it has been proposed to model each data object (e.g., a point cloud) as a metric measure space $(X,d_X,\mu)$, i.e., consisting in a metric space endowed with a (probability) measure. A particular example would be a graph with weighted vertices and edges inducing a distance. In that case, $X$ is then the set of points/nodes, $d_X$ provides the distance function between the latter points (e.g., the weighted shortest-path distance)  and $\mu$ encodes the ``importance weight" associated to each node.
\bgroup
As a concrete example we mention a protein structure (see Section~\ref{sec: Num} for details) represented by its chain of carbon atoms. This can be modelled as a metric measure space. The locations\footnote{The location is estimated indirectly via the electrons density. We refer to \cite{rhodes2010crystallography} for more details.} of these atoms in the three-dimensional space constitute the nodes and this set of points is metrised with the Euclidean distance. The mass is either chosen uniform as in \cite{weitkamp2020gromov} or proportional to the inverse of the temperature factor, as the temperature factor encodes the uncertainty in the measurements.
\egroup
In the sequel, we will often abbreviate metric measure space to mm-space as originating in \cite{gromov1999metric}.

According to \cite[p.117]{gromov1999metric}  two metric measure spaces $(X,d_X,\mu)$ and $(Y,d_Y,\nu)$ are said to be \textit{isomorphic} if  
\begin{equation}
\label{eq: Isomorph}
\exists \text{ a bijection } \psi : X \to Y : 
\begin{cases}
\psi_{\#} \mu = \nu \\
d_Y( \psi (x), \psi(y)) = d_X (x,y),
\end{cases} 
\end{equation}
where $\psi_{\#} \mu$ is the pushforward measure of $\mu$ by $\psi$.
This means, that in addition to being an isometry, $\psi$ must transport the mass of $\mu$ to $\nu$.
In this respect, a very convenient distance between two such spaces --- after ``quotienting out''  isomorphisms --- is the Gromov--Wasserstein distance \cite{Memoli2007} defined %
as 
\begin{equation}
\label{eq: GW}
 GW_p(X,Y):= \left\{ \inf_{\pi \in \Gamma(\mu,\nu)}\int_{X\times Y}\int_{X\times Y} \lvert d_X(x,x') - d_Y(y,y') \rvert^p \diff \pi(x,y) \diff \pi(x',y')  \right\}^{1/p},
\end{equation}
where $ \Gamma(\mu,\nu)$ is the set of couplings of $\mu$ and $\nu$. Even though appealing, the practical use of the concept is limited by computational difficulties. Indeed, in general, \eqref{eq: GW} constitutes a quadratic program and is as such NP-hard \cite{pardalos1991quadratic}. In practice, one then resorts either to regularised versions of the problem \cite{peyre2016gromov,sejourne2021unbalanced} or to lower bounds \cite{Memoli2007,liebscher2018new,memoli2018distance,memoli2021ultrametric,weitkamp2020gromov} on the latter distance. In the first case, there is no general proof of convergence of the proposed projected gradient descent and the quantity returned is not necessarily the optimum in \eqref{eq: GW}. In general, the geometry of $GW_p$ is theoretically poorly understood from a computational point of view. In the second case, the lower bounds can be seen as measuring a distance between certain characteristics (or signatures) of the mm-spaces. However, this often comes with some loss of information.  In particular, the fact that these lower bounds are zero does not imply that the mm-spaces are isomorphic. Still, such signatures can convey important information and help discriminate between two mm-spaces, see \cite{gellert2019substrate} for an application in structural chemistry and \cite{weitkamp2020gromov} for an application in protein matching.

In this paper, we will take an alternative route and focus on the $\rho$-Spectrum, a signature related to the spectrum of a family of Laplace operators indexed by a parameter $\rho$ introduced in \cite{burago2019spectral}.  We discuss its properties, show how this spectrum relates to other mm-spaces descriptors such as lower bounds on the $GW_p$ distance and derive statistical tools for inference on the $\rho$-Spectrum.
Contrarily to computing the Gromov--Wasserstein distance, computing the $\rho$-Spectrum can be performed in less than $\Oh(K^5)$ operations, given an observed mm-space with $K$ nodes, see Subsection \ref{sec: CompIss}. 

We begin by a recap on classical graph Laplacians in Section~\ref{sec: ClasGraphLapl} and then introduce the $\rho$-Laplace operators in Section~\ref{sec: Lapl} along with some of their properties. We show various potential uses of quantities based on this mm-space Laplacian in Section~\ref{sec: Num} and discuss a series of (numerical) examples. In particular, in Section~\ref{sec: Harmonics}, we illustrate the discriminative power of the $\rho$-Laplacian for a network analysis of the temporal evolution of microtubule cell networks. A link to semi-supervised learning problems is further developed in Section~\ref{sec: SSL}. 
In Section~\ref{sec: Inf}, we provide some foundations for statistical inference and analyse the empirical mean spectrum by providing a (bootstrap) central limit theorem and concentration bounds. As an application, confidence bands for the second largest (Fiedler) eigenvalue as a function of $\rho$ are discussed an applied to animal shape data.
\vspace{12pt}

\textit{Notation.}  We use the notation $\e_i$ for the $i$-th basis vector in $\reals^n$, i.e., the vector filled with 1 at the position $i$ and zero elsewhere. The operator $\Lambda(\cdot)$ from the set of real, symmetric $n\times n$ matrices to $\reals_+^n$ is the operator returning the eigenvalues in decreasing order. In the context of a graph, considering an edge $e$ between the neighbours $j$ and $\ell$, $in(e)=j$ and $out(e)=\ell$. As we will consider only undirected graphs in the sequel, any arbitrary traversing order is valid and unimportant. We will also use $d_e$ to mean $d_{in(e),out(e)}$. For a space $\mathcal{S}$, $\mathcal{P}(\mathcal{S})$ denotes the set of Borel probability measures on $\mathcal{S}$. For a metric space $(X,d)$, a  set $A\subset X$ and some $r\geq0$, $A^r$ denotes the closed $r$-neighbourhood of A, i.e., $A^r:=\{x \in X: d(x,A)\le r\}$. The set of distance matrices of size $K\times K$ will be denoted by $S_+^K$.  Finally, $\llbracket K \rrbracket$ denotes  $\{1,\ldots, K\}$.
\section{Classical graph Laplacian}
\label{sec: ClasGraphLapl}

Given a weighted graph $G =(V, E, W)$, with $W$ the weighted adjacency matrix, let  $\operatorname{Deg}$ denote the diagonal degree matrix of the edges. 
We use the notation $j\sim \ell$ to mean that vertices $j$ and $\ell$ are connected. We do not allow a vertex to be connected with itself. Then we recall that
$$\operatorname{Deg}_{j,\ell}=
\begin{cases} \sum_{j'} W_{j,j'} \quad\text{ if } j= \ell
\\ 0 \qquad \text{    otherwise},
\end{cases}
$$
where
$$
\begin{cases} W_{j,\ell}>0 \quad\text{ if }j\sim \ell,
\\ W_{j,\ell}=0 \qquad \text{    otherwise}.
\end{cases}
$$
The graph Laplacian $L$ (see e.g., \cite{spielman2010algorithms}) is defined as $\operatorname{Deg}-W$. The first (read smallest) eigenvalue of the graph Laplacian is null with corresponding eigenvector $\mathbf{1}$. It is a well known fact that if the graph is connected, the second smallest eigenvalue is strictly positive.
There further is a particularly convenient way of writing a graph Laplacian as a sum over all edges $e$ of size-two Laplacians, i.e.,
\begin{equation}
\label{eq: ConvProp}
L = \sum_{e \in E} W_{in(e),out(e)} b_e b_e^\top,
\end{equation}
where $b_e:= \e_{in(e)} - \e_{out(e)}$ (see the Introduction). %

More importantly for our purposes, we note that the Laplacian of a finite graph helps characterise the latter and its spectrum is an invariant of the graph. As for the lower-bounds touched upon in Section~\ref{sec: Intro}, the signature does uniquely identify a graph but there are non-isomorphic graphs with the same signature; i.e., graphs with the same Laplace spectrum that are different. This situation is depicted on Fig.~\ref{fig: FigIso}.

The eigenvalues of a connected graph Laplacian  $L \in \reals^K \times \reals^K$, $0=\lambda_1<\lambda_2 \leq \ldots\leq\lambda_{K}$ also have a useful interpretation. For instance, the second smallest eigenvalue \textemdash also called Fiedler eigenvalue \textemdash is a solution to the constrained optimisation problem
\begin{equation}
\label{eq: eigValLapl}
\min_{u \in \reals^K : \bar{u}=0, u^\top u=1 } u^\top Lu =  \min_{u \in \reals^K : \bar{u}=0, u^\top u=1 } \sum_{e \in E} W_{in(e),out(e)} ( u_{in(e)}-u_{out(e)} )^2 ,
\end{equation}
where $\bar{u}=0$ is the arithmetic mean of the entries of the vector $u$.
This corresponds to finding a one-dimensional embedding of the graph on the line that ought to respect the edge lengths of the graph.
We refer to Fig.~1 in \cite{levy2006laplace} for a visualisation of the fact that this vector captures the ``directionality'' of the shapes approximated by a mesh. In a sense, the eigendecomposition of the Laplacians ``captures the harmonics of the shape''. If the graph is not connected, the number of zero eigenvalues will provide the number of disconnected components.
\bgroup

Also, it is important to say that \eqref{eq: eigValLapl} is a relaxation of the minimal ratio cut problem, i.e., the problem of splitting the graph into two parts of roughly the same size with the smallest connectivity possible\footnote{To be precise, the objective there is to solve 
\[
 \min_{u \in \{0,1\}^K } \tr \left( L \left( \frac{uu^\top}{ \lvert u \rvert} + \frac{(\mathbf{1}-u)(\mathbf{1}-u)^\top}{K - \lvert u \rvert} \right)\right).\]
 }.
In the latter problem, the vectors $u$ must belong to $\{0,1\}^K$ before being renormalised. 
The eigenvector is thus an approximation of the vector giving the partition in minimal ratio cut problem. Often, one heuristically chooses the point where the eigenvector's sign changes as the splitting point even though other methods exist \cite[Section~5.1]{von2007tutorial}. 

The largest eigenvalue can also be better understood through the well-known following result. 
\begin{lemma}
\label{lem: largestEig}Let $L$ be the graph Laplacian of a graph with $K$ nodes, then 
\[
\lambda_{\max}(L) \geq \max_{j \in \llbracket K \rrbracket} \operatorname{Deg}_{j,j}.
\]
\end{lemma}
In \cite{berman2011lower}, they further prove that, for a graph with more than 3 vertices, the second largest eigenvalue of the Laplacian is bounded below by the third largest degree.

\egroup

With these elements in mind, we can now go back to the subject of interest: mm-spaces.
\section{The $\rho$-Laplacian for a metric measure space}
\label{sec: Lapl}

\input{DefProp}

\section{ $\rho$-Laplacian for data analysis}
\label{sec: Num}

\input{Applications}
\input{Inference}

\section*{Acknowledgments}
Both authors are grateful to S.~K\"oster, A.~Blob, and D.~Ventzke for the preparation and preprocessing of some data used in this paper. The first author is also thankful to C.~A.~Weitkamp for numerous discussions.

\newpage
\bibliographystyle{siamplain}
\bibliography{Ref}

\end{document}

%% file: ex_shared.tex
\usepackage{lipsum}
\usepackage{amsfonts}
\usepackage{graphicx}
\usepackage{epstopdf}
\usepackage{algorithmic}
\ifpdf
  \DeclareGraphicsExtensions{.eps,.pdf,.png,.jpg}
\else
  \DeclareGraphicsExtensions{.eps}
\fi

\usepackage{enumitem}
\setlist[enumerate]{leftmargin=.5in}
\setlist[itemize]{leftmargin=.5in}

\newsiamremark{rmk}{Remark}
\newsiamremark{ex}{Example}
\newsiamremark{hypothesis}{Hypothesis}
\crefname{hypothesis}{Hypothesis}{Hypotheses}
\newsiamthm{claim}{Claim}

\headers{Statistical analysis of random objects via  metric measure Laplacians}{G. Mordant, and A. Munk}

\title{Statistical analysis of random objects via  metric measure Laplacians
}

\author{Gilles Mordant\thanks{Institut f\"ur Mathematische Stochastik, Universit\"at G\"ottingen.\funding{G. Mordant gratefully acknowledges the funding by the DFG SFB 1456-A04.}}
\and Axel Munk{\footnotemark[1] $^,$\thanks{Max Planck Institute for Biophysical Chemistry and DFG Cluster of Excellence ``Multiscale Bioimaging : from Molecular Machines to Networks of Excitable Cells'' (MBExC), University Medical Center, G\"ottingen.} }
}

%% file: DefProp.tex
Let us place ourselves in a context where the observations at hand are not graphs anymore but mm-spaces. %
\subsection{Definition and associated Dirichlet form}
In light of Section~\ref{sec: ClasGraphLapl}, %
we can now define the convolutional Laplacian first introduced in \cite{burago2019spectral} and exhibit some of its properties.
\begin{definition}
Let $(X,d,\mu)$ be a metric-measure space and $\rho>0$. Then, the $\rho$-Laplacian $\Delta_X^\rho: L^2(X,\mu)\to L^2(X,\mu)$ is defined by
\[
\Delta_X^\rho u(x):= \frac{1}{\rho^{2} \mu (B_\rho(x ))} \int_{B_\rho(x )} \big(u(x) -u(y) \big) \mathrm{d}\mu(y), 
\]
for $u \in  L^2(X,\mu)$ and where $B_\rho(x)$ is the metric ball of radius $\rho$ centred at $x\in X$. 
\end{definition}
The corresponding Dirichlet form is then 
\[
D_X^{\rho}(u):= \frac{1}{2} \iint_{d_X(x,y)<\rho} \big(u(x)-u(y)\big)^2 \mathrm{d}\mu(x) \mathrm{d}\mu(y).
\]
In the case where the mm-space under consideration is discrete, i.e., $\mu= \sum_j \mu_j\delta_{x_j}$ with $\mu_j >0$ for all $j \in \llbracket K \rrbracket$,  the Dirichlet form can be rewritten as
\begin{equation}
\label{eq: DirichletForm}
D_X^{\rho}(u)= \frac{1}{2} \sum_{j,\ell} \big(u(x_j)-u(x_\ell)\big)^2 \mu_j\mu_\ell \1_{\{d_X(x_j,x_\ell)<\rho\}}.
\end{equation}
This is exactly the same formula as the definition of the Laplacian quadratic form\footnote{Some definitions do not involve the constant 1/2. As this is merely a scaling factor, it is of little importance.} as in Eq.~\eqref{eq: eigValLapl} of an auxiliary weighted graph with weights  $W_{j,\ell}$ 
\[
W_{j,\ell}= \mu_j\mu_\ell \1_{\{d_X(x_j,x_\ell)<\rho\}},
\]
A representation of a metric measure space with a corresponding auxiliary graph is provided in Fig.~\ref{fig: AuxGraph}.
The $\rho$-Laplacian thus corresponds to the classical graph Laplacian of this auxiliary graph. We now provide one possible interpretation of the weights of the auxiliary graph. 

\begin{rmk}[ ``Gravitational'' interpretation of the auxiliary graph]
The weights of the auxiliary graph can be seen as a gravitational type of interaction between the nodes in the mm-space. Indeed, the interaction is proportional to the product of the masses and depends on the distance.  In this case, the node of the auxiliary graph are inherited from the mm-space. The nodes will then be connected with a strength proportional to their mass provided that the distance between the two was smaller than the parameter $\rho$.
\end{rmk}

\begin{figure}
\centering
\begin{tikzpicture}            
   \node[shape=circle,fill,draw=black, scale=1.75] (A) at (0,0) {};
    \node[shape=circle,fill,draw=black, scale=1] (B) at (2.2,0) {};
    \node[shape=circle,fill, draw=black, scale=1 ] (C) at (2.1,1.3) {};
    \node[shape=circle,fill,draw=black,scale=1/1.75] (D) at (-1.3,0.9) {};

    \path [dashed] (A) edge node[above] {} (C);
    \path [-](A) edge node[below] {} (B);
    \path [-](D) edge node[left] {} (A);
    \path [-](B) edge node[right] {} (C);
         \path [dashed](C) edge node[above] {} (D);
           \path [dashed](B) edge node[right] {} (D);        
            \end{tikzpicture}
\qquad \qquad
\begin{tikzpicture}            
     \node[shape=circle,fill,draw=black] (A) at (0,0) {};
    \node[shape=circle,fill,draw=black] (B) at (2.2,0) {};
    \node[shape=circle,fill, draw=black ] (C) at (2.1,1.3) {};
    \node[shape=circle,fill,draw=black] (D) at (-1.3,0.9) {};

    \path [-,thick](A) edge node[below] {} (B);
    \path [-](D) edge node[left] {} (A);
    \path [-](B) edge node[right] {} (C);
     
\end{tikzpicture}
\caption{  \slshape\small Left) A discrete metric measure space with non uniform probability measure. The length of the vertices encode the distances. In dashed are the distances larger than the $\rho$ chosen for this example. Right) The representation of the auxiliary graph for the same $\rho$ as on the left hand side. Only the nodes close enough are connected and the strength of the link depends on the measure of the node in the mm-space. The strength is represented through different thicknesses of the edges. The edges' length however does not bear any meaning.}
\label{fig: AuxGraph}
\end{figure}
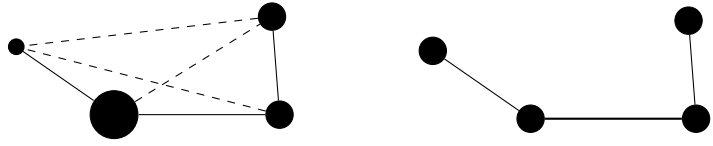

\subsection{Properties}
As advanced by \cite{burago2019spectral}, the underlying Laplacian is stable with respect to metric-measure perturbations. More precisely, if two metric measure spaces are \textit{relative $(\varepsilon, \delta)$-close} (see Definition~\ref{dfn: edClose} below) and additional technical conditions hold, then the lower parts of the spectra of their $\rho$-Laplacians are also close. We refer to Theorem~5.4. in the aforementioned reference for a precise statement. 

\begin{definition}{ \cite[Definition~4.2.]{burago2019spectral}}
\label{dfn: edClose}
Let $\varepsilon,\delta\geq 0.$ We say that two mm-spaces $X=(X,d_X,\mu_X)$ and $Y=(Y,d_Y,\mu_Y)$ are mm-relative $(\varepsilon, \delta)$-close if there exists a semi-metric $d$ on the disjoint union $X \sqcup Y$, extending $d_X$ and $d_Y$, such that  
\[
e^\delta \mu_X(A^\varepsilon) \geq \mu_Y(A) \qquad \text{ and } \qquad e^\delta \mu_Y(A^\varepsilon) \geq \mu_X(A),
\]
for every Borel set $A\subset X \sqcup Y$, where the closed neighbourhood is constructed with respect to the metric $d$.
\end{definition}

Before turning to the spectrum, we recall two important facts about the relative $(\varepsilon, \delta)$-closeness. The latter are stated in \cite{burago2019spectral} and are helpful to relate the proposed concepts to the Gromov-Wasserstein distance. 
\begin{rmk}
$(\varepsilon, 0)$-closeness of mm-spaces is equivalent to the $\varepsilon$-closeness with respect to  $L_\infty$ Gromov-Wasserstein distance.
\end{rmk}
\begin{rmk}
Consider an mm-space $Y' =(Y, d_{Y'}, \mu_{Y'})$ that is a ``blurred'' version of $Y$, in the sense that $\lvert d_{Y'}-  d_{Y}\rvert <\varepsilon$ and 
$e^{-\delta} \le \mu_{Y'} / \mu_{Y} \le e^{\delta}$. Then $Y'$ is $(2\varepsilon, \delta)$ close to $Y$. 
\end{rmk}
\begin{rmk}
If the mm-space under consideration is a Riemannian manifold endowed with its geodesic distance and volume measure, the $\rho$-Laplacian above converges to a constant times the Laplace--Beltrami operator as $\rho \to 0^{+}$.
\end{rmk}
We finish this section by establishing a desirable property of the Laplacian. Before doing so, we recall two important elements from \cite{burago2019spectral}. Setting
\[
\lVert u \rVert_{X^\rho}^2 := \rho^2 \int_X \mu( B_\rho(x)) u^2(x) \mathrm{d}\mu(x),
\]
the eigenvalues of $\Delta_X^\rho$ are defined by the Courant-Fischer-Weyl principle, that is, for $k \in \mathbb{N}$,
\[
\lambda_k(X, \rho) = \inf_{H_{X}^k} \sup_{u \in H_X^k \setminus\{0\} }\  \frac{D_X^{\rho}(u)}{\lVert u \rVert_{X^\rho}^2},
\]
where $H_{X}^k$ is a $k$ dimensional subspace of $L^2(X,\mu)$ and the first infimum is taken over all such $H_{X}^k$.
\begin{proposition}
For two isomorphic metric measures spaces \textemdash recall~\eqref{eq: Isomorph},  the spectra coincide. 
Further, the if $u'$ is a $\rho$-eigenfunction of $(Y,d_Y,\nu)$, then $u' \circ \psi$ is a $\rho$-eigenfunction of $(X,d_X,\mu)$.
\end{proposition}
\begin{proof}
From the definition of isomorphism, one has 
\begin{align*}
D_Y^{\rho}(u') &= \frac{1}{2} \iint_{d_Y(x,y)<\rho} (u'(x)-u'(y))^2 \mathrm{d}\nu(x) \mathrm{d}\nu(y)\\
& = \frac{1}{2} \iint_{d_Y( \psi (x), \psi(y))<\rho} (u'(\psi(x))-u'(\psi(y)))^2 \mathrm{d}\mu(x) \mathrm{d}\mu(y)\\
& =\frac{1}{2} \iint_{d_X(x,y)<\rho} (u'(\psi(x))-u'(\psi(y)))^2 \mathrm{d}\mu(x) \mathrm{d}\mu(y)\\
& = D_X^{\rho}(u'\circ \psi).
\end{align*}
\bgroup
Using the same type of development, one has $\lVert u' \rVert_{Y^\rho}^2 = \lVert u' \circ \psi \rVert_{X^\rho}^2$. Further, each element  $u' \in H_Y^k\setminus\{0\}$ can be identified with an element $u' \circ \psi \in H_X^k\setminus\{0\}$. Indeed, one can develop each function $u'$ in its $L^2(Y,\nu)$-basis. From the latter basis, one can construct a basis for $L^2(X,\mu)$ by composition of each basis element with $\psi$, as the orthonormality conditions will trivially hold in $L^2(X,\mu)$ by  the change of variable formula
\[
\int f \diff (\psi_{\#} \mu) = \int (f \circ \psi)  \diff \mu. 
\]
By the same formula, the coefficients of $u'$ and $u' \circ \psi$, in their respective expansion, match. 
This ensures that the eigenvalues obtained from the Courant-Fischer-Weyl formulation are identical for $(Y,d_Y,\nu)$ and $(X,d_X,\mu)$. The claims follow.
\egroup
\end{proof}
\bgroup
\subsection{Euclidean embedding of mm-spaces}
\label{sec: EucEmb}
At this stage, it is illustrative to comment on the embedding of an mm-space indexed by $\rho$.
Importantly, one can see the spectra as a natural quantity arising from the construction of a $\rho$-dependent embedding of the mm-space into $\reals^K$. 
Depending on the application, one will focus more on certain characteristics of the embedding, as one shall see.
This embedding occurs in two stages, first through the construction of the auxiliary graph and then through the embedding of the latter graph into $\reals^K$. 
In this context, we recall that, starting from the Laplacian of the auxiliary graph for a given $\rho$, the points are defined by 
\[
\mathbf{X}_{em} := \operatorname{diag}[(\lambda_2^\rho, \ldots, \lambda_K^\rho)]^{-1/2} U_\rho^\top,
\] 
where $U_\rho$ is the matrix of eigenvectors of the Laplacian of the auxiliary graph.
Therefrom, one easily sees that the inverse eigenvalues provide the variability of the point cloud $\mathbf{X}_{em}$ along its principal directions given by the eigenvectors of the auxiliary graph.

\egroup
\subsection{Comparison with other shape descriptors}
\label{sec: Ex}
In the sequel, we will propose inter alia to rely on a spectrum of the $\rho$-Laplacian to carry out inference on mm-spaces. It is therefore natural to assess the discriminative power of the spectra and compare it with other descriptors of mm-spaces. To do so, in this section, we compute the spectra for various $\rho$ parameters in various examples of mm-spaces that are presented in \cite{memoli2011gromov}. 
\bgroup
We begin with two examples that illustrate the behaviour of the $\rho$-spectra. Then, we exhibit that, contrarily to certain lower bounds for the GW-distance,
the $\rho$-Spectrum can discriminate between the spaces in Examples \ref{ex: Fig7Mem}, \ref{ex: Plan} and \ref{ex: Tree}.
These lower bounds are the distribution of distances (DoD)
\[
t \mapsto \mu \otimes \mu \{ (x,x') : d_X(x,x') \le t \} ,
\] 
and the local distribution of distances, i.e.,
\begin{align} 
\label{eq: locDist}
&h_X : X \times [0,\operatorname{diam}(X)] \to [0,1] \\ \nonumber
&\quad(x,t) \mapsto \mu \left( \{x' : d_X (x,x') \le t \}\right).
\end{align}
\egroup
\begin{ex}
Take two two-point spaces with distance 1 between the two points and masses $\{3/4,1/4\}$ and $\{1/2,1/2\}$, respectively. 
Then, for $\rho>1$, their spectra are $\{0,3/8\}$ and $\{0,1/2\}$, respectively.  

In general, for any two-point space with total mass $\mathfrak{m}$, mass of the first point $a$ and distance between the two points $r$, the $\rho$-spectrum is given by $\{0, 2a (\mathfrak{m} -a)\1_{\{r <\rho\}}\}$. 
\end{ex} 
\begin{ex}
Define $\{\Delta_n\}_{n=2}^\infty$ to be the sequences of mm-spaces such that $\Delta_n$ has $n$ points, the distance between any pair of points of   $\Delta_n$ is one and the mass associated to each vertex is $1/n$. 
Then, the sequence of spectra is $ \{0,  \1_{\{\rho\geq 1\}}\mathbf{1}_{n-1}^\top /n\}$.
\end{ex} 

\begin{ex}
\label{ex: Fig7Mem}
Consider the two mm-spaces in Fig.~\ref{fig: SignatureExample}, top panel. They are not isomorphic but have the same distribution of distances. 
Set\footnote{This is just larger than the smallest distance between two points.} $\rho=1.5$, 
the spectra are $\{0.1875,0.625,0,0\}$ for the right hand side mm-space on the aforementioned graph and $\{0.125,0.125,0,0\}$ for the second one.
For $\rho=2.5$, the spectra are $\{0.1875, 0.1875, 0, 0\}$ and $\{0.2134, 0.1250, 0.0366, 0\}$, respectively. A representation of the $\rho$-spectra is shown on Fig.~\ref{fig: SignatureExample}. One clearly sees that the distances and connections are taken into account. On the left hand-side, the fact that the two largest eigenvalues are identical after the first jump exhibits the fact that they are two similar, disconnected substructures. After the second jump, even though the distances are the same, the geometric structure is a triangle on the right hand side while it is a path on the left hand side. This is also reflected by a difference between the spectra.
\end{ex} 
\begin{ex}
\label{ex: Plan}
 Let $X =(0,1,4,10,12,17)$ and $Y=(0,1,8,11,13,17)$ be two sets of points. Endow them with the Euclidean distance on $\reals$ and a uniform probability measure. They are non-isomorphic but have the same distributions of distances. 
If one sets $\rho=2.5$, both spectra are the same and $\{0.556, 0.556,0,0,0,0\}$. If one chooses $\rho=4.5$, a discrimination becomes possible as the spectra are now given by $\{0.0833, 0.0833, 0.0556, 0, 0, 0\}$ for $X$ and $\{0.0948, 0.0556, 0.0556, 0.0163, 0, 0\}$.
\end{ex}
\begin{ex}
\label{ex: Tree}
We finally consider the two mm-spaces in \cite[Fig.~8]{memoli2011gromov}. In this last example, one obtains the set of eigenvalues 
$\{0.05396, 0.04635, 0.04444, 0.02667, 0.01573, 0.00463, 0, 0, 0\}$ for the space $X$ and $\{0.05234, 0.04915, 0.04361, 0.02613, 0.01592, 0.00462, 0, 0, 0\}$ for $Y$, choosing $\rho=1.5$. The two spectra are close but different. On the other hand, for $\rho=2.5$, the spectra are the same, which is expected from the apparent permutations possible. However these permutations do not give rise to an isometry.  
 \end{ex}
From these examples, the conclusion is that the $\rho$-varying spectra can help discriminate between mm-spaces for which certain lower bounds on the Gromov-Wasserstein distance were not sensitive enough. In this matter, the presence of the $\rho$-parameter plays an important role as was seen on the last example, for instance.

\begin{figure}
\begin{center}

\centering
\begin{tikzpicture}            
    \node[shape=circle,draw=black] (A) at (0,0) {\footnotesize A};
    \node[shape=circle,draw=black] (B) at (3.9,-1.3) {\footnotesize B};
    \node[shape=circle,draw=black] (C) at (5.2,0) {\footnotesize C};
    \node[shape=circle,draw=black] (D) at (3.9,1.3) {\footnotesize D};
   \node at (3.73,-0.65) {\footnotesize 2};

    \path [-] (A) edge node[above] {\footnotesize$4$} (C);
    \path [-](A) edge node[below] {\footnotesize$\sqrt{10}$} (B);
    \path [-](A) edge node[above] {\footnotesize$\sqrt{10}$} (D);
    \path [-](B) edge node[right] {\footnotesize$\sqrt{2}$} (C);
     \path [-](B) edge node[] {} (D);
       \path [-](C) edge node[right] {\footnotesize$\sqrt{2}$} (D);
\end{tikzpicture}
\qquad
\begin{tikzpicture}            
    \node[shape=circle,draw=black] (A) at (0,0) {\footnotesize A'};
    \node[shape=circle,draw=black] (B) at (2.6,0) {\footnotesize B'};
    \node[shape=circle,draw=black] (C) at (3.9,1.3) {\footnotesize C'};
    \node[shape=circle,draw=black] (D) at (-1.3,1.3) {\footnotesize D'};

    \path [-] (A) edge node[above] {\footnotesize$\sqrt{10}$} (C);
    \path [-](A) edge node[below] {\footnotesize$2$} (B);
    \path [-](D) edge node[left] {\footnotesize$\sqrt{2}$} (A);
    \path [-](B) edge node[right] {\footnotesize$\sqrt{2}$} (C);
     \path [-](B) edge node[above] {\footnotesize$\sqrt{10}$} (D);
       \path [-](C) edge node[above] {\footnotesize 4} (D);
\end{tikzpicture}

\begin{tikzpicture}[scale=0.75]

\begin{axis}[
    xlabel={$\rho$},
    ylabel={Spectrum},
    xmin=0, xmax=5,
    ymin=0, ymax=0.3,
    xtick={0,1,2,3,4,5},
    ytick={0,0.1,0.20,0.3},
    legend pos=north west,
]
\addplot[
    color=blue,
    mark=*
        ]
    coordinates {
(0,0)
(1.41,0)
(1.41,0.185)
(2,0.185)
(2,0.185)
(3.162,0.185)
(3.162,0.25)
(4,0.25)
(4,0.25)
(5,0.25)
    };
 \addlegendentry{$\lambda_4$}
\addplot[
    color=red,
    mark=*
        ]
    coordinates {
(0,0)
(1.41,0)
(1.41,0.0625)
(2,0.0625)
(2,0.185)
(3.162,0.185)
(3.162,0.25)
(4,0.25)
(4,0.25)
(5,0.25)
    };
 \addlegendentry{$\lambda_3$}
 \addplot[
    color=gray,
    mark=*
        ]
    coordinates {
(0,0)
(1.41,0)
(1.41,0)
(2,0)
(2,0)
(3.162,0)
(3.162,0.125)
(4,0.125)
(4,0.25)
(5,0.25)
    };
 \addlegendentry{$\lambda_2$}
\end{axis}
\end{tikzpicture}
\begin{tikzpicture}[scale=0.75]
\begin{axis}[
    xlabel={$\rho$},
    ylabel={Spectrum},
    xmin=0, xmax=5,
    ymin=0, ymax=0.3,
    xtick={0,1,2,3,4,5},
    ytick={0,0.1,0.20,0.3},
    legend pos=north west,
]
\addplot[
    color=blue,
    mark=*
        ]
    coordinates {
(0,0)
(1.41,0)
(1.41,0.125)
(2,0.125)
(2,0.2134)
(3.162,0.2134)
(3.162,0.25)
(4,0.25)
(4,0.25)
(5,0.25)
    };
 \addlegendentry{$\lambda_4$}
\addplot[
    color=red,
    mark=*
        ]
    coordinates {
(0,0)
(1.41,0)
(1.41,0.125)
(2,0.125)
(2,0.125)
(3.162,0.125)
(3.162,0.25)
(4,0.25)
(4,0.25)
(5,0.25)
    };
 \addlegendentry{$\lambda_3$}
 \addplot[
    color=gray,
    mark=*
        ]
    coordinates {
(0,0)
(1.41,0)
(1.41,0)
(2,0)
(2,0.0366)
(3.162,0.0366)
(3.162,0.125)
(4,0.125)
(4,0.25)
(5,0.25)
    };
 \addlegendentry{$\lambda_2$}
\end{axis}
\end{tikzpicture}
\caption{\label{fig: SignatureExample} \slshape\small Top panel) The two metric measures spaces in Example~\ref{ex: Fig7Mem}. Bottom panel) The corresponding spectra for the two different models of the top panel as a function of the parameter $\rho$. The eigenvalue $\lambda_1$ is not represented as it is zero for each $\rho$, which is a fundamental property of graph Laplacians. }
\end{center}
\end{figure}
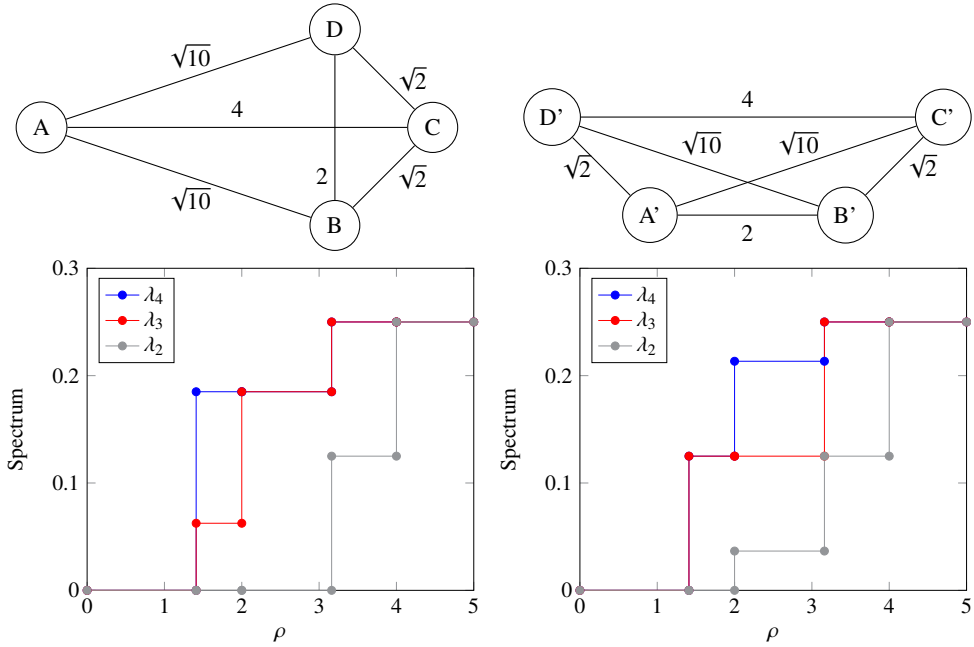

\subsection{Uniform isospectrality}
In this section, we investigate the cases where two discrete metric-measure spaces have the same spectrum uniformly in $\rho$. 
\bgroup
First, we recall that there exist non isometric graphs with the same Laplace spectrum; the two graphs on Fig.~\ref{fig: FigIso} provide an example\footnote{The entries of the adjacency matrix are fixed to a constant when there is an edge and to zero otherwise.}. The latter are sometimes called \textit{isospectral} and the question to understand when isospectrality occurs for non isometric graphs has an already long history. It is not fully solved, albeit some progress has been made, see for instance \cite{qiu2020theorem,van2003graphs, brouwer2011spectra,halbeisen2000reconstruction} and the references therein.
 
\begin{figure}
\centering
\begin{tikzpicture}            
    \node[shape=circle,draw=black, fill=black,inner sep=0pt,minimum size=8pt] (A) at (0,0) {};
      \node[shape=circle,draw=black, fill=black,inner sep=0pt,minimum size=8pt]  (B) at (1,0) {};
        \node[shape=circle,draw=black, fill=black,inner sep=0pt,minimum size=8pt]  (C) at (2,0) {};
        \node[shape=circle,draw=black, fill=black,inner sep=0pt,minimum size=8pt] (D) at (3,0) {};
       \node[shape=circle,draw=black, fill=black,inner sep=0pt,minimum size=8pt]  (E) at (1.5,-1) {};
       \node[shape=circle,draw=black, fill=black,inner sep=0pt,minimum size=8pt]  (F) at (1.5,1) {};

    \path [-] (A) edge node[above] {} (F);
    \path [-](A) edge node[below] {}  (E);
     \path [-](E) edge node[] {} (D);
      \path [-](F) edge node[] {} (D);
        \path [-](F) edge node[] {} (C);
          \path [-](F) edge node[] {} (B);
       \path [-](C) edge node[right] {}  (B);
\end{tikzpicture}
\qquad
\begin{tikzpicture}            
    \node[shape=circle,draw=black, fill=black,inner sep=0pt,minimum size=8pt] (A) at (0,0) {};
      \node[shape=circle,draw=black, fill=black,inner sep=0pt,minimum size=8pt]  (B) at (1,0) {};
        \node[shape=circle,draw=black, fill=black,inner sep=0pt,minimum size=8pt]  (C) at (2,0) {};
        \node[shape=circle,draw=black, fill=black,inner sep=0pt,minimum size=8pt] (D) at (3,0) {};
       \node[shape=circle,draw=black, fill=black,inner sep=0pt,minimum size=8pt]  (E) at (1.5,-1) {};
       \node[shape=circle,draw=black, fill=black,inner sep=0pt,minimum size=8pt]  (F) at (1.5,1) {};

    \path [-] (A) edge node[above] {} (F);
    \path [-](A) edge node[below] {}  (E);
     \path [-](E) edge node[] {} (D);
      \path [-](F) edge node[] {} (D);
       
        \path [-](B) edge node[] {} (E);
          \path [-](F) edge node[] {} (B);
       \path [-](C) edge node[right] {}  (B);
\end{tikzpicture}
\caption{ \slshape\small Two isospectral graphs.}
\label{fig: FigIso}
\end{figure}
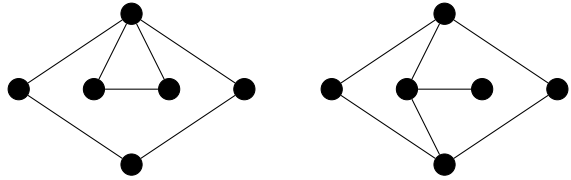
The case that interests us contrasts with the literature. Instead of a single graph, we are considering a $\rho$-indexed family of graphs. Completely characterizing isospectrality uniformly in $\rho$ is certainly beyond the horizon. Still, we now aim at understanding what information can be retrieved from the knowledge of the fact that $\forall \rho, \Lambda(L_1(\rho))=\Lambda(L_2(\rho))$.
\egroup

\begin{proposition}
If two finite mm-spaces with edge sets $E_1$, $E_2$ and weight matrices $W_1(\rho)$, $W_2(\rho)$ respectively are isospectral uniformly in $\rho$, then the two sets of distances are the same. 
Further, the weights need to be such that 
\[
 \sum_{e\in E_1} W_{1,e}(\rho)    =  \sum_{e\in E_2} W_{2,e}(\rho), \forall \rho.
\]
\end{proposition}
\begin{proof}
First, we recall the fact that two matrices $A,B$ of size $n\times n$ have the same spectrum if and only if $\trace(A^i) =\trace(B^i)$ for $i \in \llbracket n \rrbracket$, see Lemma1 in \cite{van2003graphs}. 
Recalling the formula for Laplacians in \eqref{eq: ConvProp}, we thus have that $\Lambda(L_1^\rho)=\Lambda(L_2^\rho), \forall \rho,$ implies
\begin{align*}
\forall \rho, \trace(L_1^\rho) =\trace(L_2^\rho) 
&\Rightarrow \trace \left( \sum_{e\in E_1} W_{1,e}(\rho) b_e b_e^\top  \right) = \trace \left( \sum_{e\in E_2} W_{2,e}(\rho) b_e b_e^\top  \right), \forall \rho \\
&\Rightarrow  \sum_{e\in E_1} W_{1,e}(\rho) \trace  \left(b_e b_e^\top  \right) =  \sum_{e\in E_2} W_{2,e}(\rho) \trace  \left(b_e b_e^\top  \right), \forall \rho \\
&\Rightarrow  \sum_{e\in E_1} W_{1,e}(\rho)    =  \sum_{e\in E_2} W_{2,e}(\rho), \forall \rho,
\end{align*}
where the last implication comes from the fact that $b_e = \e_{in(e)}-\e_{out(e)}$, which implies that $\trace  \left(b_e b_e^\top  \right)=2$ for each $e$ as
$\diag(b_e b_e^\top) =\e_{in(e)}+\e_{out(e)} $.
Notice that $\sum_{e\in E_1} W_{1,e}(\rho)$ is increasing in $\rho$ so that uniform isospectrality requires that the sets of observed distances for both spaces be the same. Indeed, if it weren't the case, there would exist a certain $\rho_0$ and $\varepsilon>0$, such that $\sum_{e\in E_1} W_{1,e}(\rho_0 + \varepsilon)- \sum_{e\in E_1} W_{1,e}(\rho_0 )$ is strictly positive while $\sum_{e\in E_2} W_{2,e}(\rho_0 + \varepsilon)- \sum_{e\in E_2} W_{2,e}(\rho_0)$ is null. Further, the increase of both spectra must be of an amount  $\mu_{1, out(e)}  \mu_{1, in(e)} = \mu_{2, out(e)}  \mu_{2, in(e)}$.
\end{proof}

\begin{proposition}
\label{sec: spanTrees}
Consider two finite, $\rho$-uniformly isospectral mm-spaces with the uniform measure. Assume that $\rho_0$ is such that for all $\rho> \rho_0$, the auxiliary graph to which the $\rho$-Laplacian corresponds is connected. 
Then, for each $\rho> \rho_0$, the number of spanning trees of the auxiliary graphs $G_1^\rho$ and $G_2^\rho$ must be equal.
\end{proposition}
\begin{proof}
From Proposition 1.3.4 in \cite{brouwer2011spectra}, the number of spanning trees of a graph $G$, $\#\operatorname{ST}(G)$, is closely related to the spectrum of the graph when the latter is unweighted. Indeed, assuming that the graph Laplacian has $K-1$ nonzero eigenvalues, it holds 
\[
\#\operatorname{ST}(G)= \frac{1}{K} \prod_{j=2}^K  \lambda_j.
\]
As considering a uniformly weighted graph is equivalent up to constant as considering unweighted graphs, uniform isospectrality for mm-spaces with the uniform measure implies that the number of spanning trees behaves exactly in the same way for the two graphs as a function of the location parameter $\rho$ as soon as they are both connected. \end{proof}
One can also check that the local distribution of distances, Eq.~\eqref{eq: locDist},
is encoded in the Laplacian. Therefore, uniform isospectrality entails that the local distribution of distances can not be too different. %
The $(i,i)$-th element on the diagonal of the  $\rho$-Laplacian can be rewritten as a function of the local distribution of distances, 
\[
\sum_{j=1; j\neq l}^K \mu_j\mu_\ell \1_{\{d(x_j,x_\ell)\le\rho\}} = \mu_j (h_X(x_j, \rho)-\mu_j),
\]
\bgroup
assuming for simplicity that that no pair of points is at a distance exactly equal to $\rho$.
Further, from Lemma~\ref{lem: largestEig} and Lemma~\ref{lem: BoundEigGraph} \textemdash which will be introduced later on, 
we obtain that
\begin{equation}
\label{eq: BndsLocDist}
 \max_{j \in \llbracket K\rrbracket} \mu_j (h_X(x_j, \rho)-\mu_j) \le \lambda_{\max}(L^\rho) \le  2\max_{j \in \llbracket K\rrbracket} \mu_j (h_X(x_j, \rho)-\mu_j).
\end{equation}
The largest eigenvalue must thus be driven by the distribution of distances.
Therefore, from the bounds in \eqref{eq: BndsLocDist}, isospectrality entails necessary conditions on the local distribution of distances because $\lambda_{\max}(L_1^\rho)$ must be equal to $\lambda_{\max}(L_2^\rho)$.
\egroup

Let us now try to reconstruct the graphs from the spectrum in a forward manner to provide a better understanding of the information contained in the sequence of spectra.
To this end, to simplify the process, we consider a discrete mm space with uniform measure. We also make the assumption that all the distances are different. 
Looking at the spectra for increasing $\rho_1<\rho_2<\ldots $, the spectrum will show that there exist a pair of nodes that are distant from each other by a certain $\rho_1$. 
Then, two possibility arise. Either the next edge is connected to the existing structure or it is not. In the latter case, one would only have an additional eigenvalue equal to $2/K^2$.
Further, we also can ensure that the two endpoints of the added node must be outside the union of two circles, each centred on the endpoint of the first edge, of radius $\rho_2$.
In the other case, as the first non-zero eigenvalue changed, we know that the additional edge must be connected to the first structure. 

Let us now consider, still in the case of uniform weights, that we start from a configuration where the next additional edge will be linked to a two-edge structure. 
Then, we will be able to distinguish whether it is connected to the central point or to one of the endpoints. 
Indeed, in the first case, the Laplacian matrix is given by a multiple of 
\[
\begin{bmatrix}
1 & -1 & 0 & 0 \\
-1 &2  & -1 & 0 \\
0& -1 & 2 & -1 \\
0 & 0 & -1 & 1 \\
\end{bmatrix}
\] 
whose eigenvalues are well known to be $\{2 +\sqrt{2},2,2 -\sqrt{2},0\}$, 
while in the second case, the matrix is 
\[
\begin{bmatrix}
1 & -1 & 0 & 0 \\
-1 &3  & -1 & -1 \\
0& -1 & 1 & 0 \\
0 & -1 & 0 & 1 \\
\end{bmatrix}
\] 
and its eigenvalues are $\{4,1,1,0\}$. 
In these first, small-structure examples, one can derive from the spectrum information about where the nodes of the mm-space can and cannot be. Further, the the information that comes for the subsequent $\rho$'s in the sequence will also give additional structural information. One should also not forget that, as the radii under consideration increase, some position of nodes might become impossible due to the presence of other connected components in the sequence of graphs.
When the number of points increases, the number of cases to consider increases as well and any attempt of reconstruction rapidly becomes intractable. This fact is also true for metric measure spaces with non-uniform measure.

In a backward manner, when starting from a situation where the spectrum has only one null eigenvalue, the fact that removing an edge produces a spectrum with an additional null eigenvalue implies that this edge was the last  link between two components.

%% file: Applications.tex
\subsection{Computational issues}
\label{sec: CompIss}
We begin with a discussion of the computational complexity and algorithms to compute the $\rho$-Laplacian for a given dataset.
Given an observed mm-space with $K$ nodes, the first step required is to sort the upper triangle of the matrix of distances, which induces a cost of $\Oh(K^2 \log(K))$.
To retrieve the spectrum for a given $\rho$, a computational cost of $\Oh(K^\omega)$ arises, where $\omega$ is the computational complexity of matrix multiplication\footnote{Even though the theoretical complexity of matrix multiplication is such that $\omega \approx 2.37$, the complexity that one can achieve in practice is of the order $\Oh(K^{\log_2 7})$, $\log_2 7\approx2.807$, which is the complexity of Strassen's algorithm \cite{strassen1969gaussian}.}, see \cite{demmel2007fast}.  
Computing the spectrum for each $\rho$ thus requires a number of operations which is $\Oh(K^{4.807})$ but would be $\Oh(K^5)$ with a classical implementation of matrix products. As the matrices for small $\rho$ are sparse, a sparse implementation can deliver a better run time in these cases.
Finally, remark that the problem of computing the spectrum for each $\rho$ can be recast as a series of rank-one updates\textemdash see Eq.~\eqref{eq: ConvProp}. Adopting this point of view will provide a better computational complexity, when one uses first-order approximative updates as in \cite[Section~3]{cardot2018online}. A complexity of  $\Oh(K^{4})$ can thus be achieved but at the price of propagating errors, which are difficult to control. 

\subsection{Shape characterisation through the spectrum}

The first natural way of using the $\rho$-Laplacian for practical purposes is to consider the $\rho$-Spectrum as a signature of the metric measure space under consideration. This is analogous to the celebrated ``ShapeDNA'' in shape analysis \cite{reuter2006laplace}. To carry out statistical inference, we will consider the mean of these signatures in Section~\ref{sec: Inf}.

\begin{rmk}
It turns out that the form of the ``mean of signatures'' that we propose is reminiscent of mean persistence landscapes used in \textit{Topological Data Analysis} \cite[Section~4.4 for instance]{medina2016statistical}. Similarly to persistence diagrams, the spectra encode (at least) the number of connected components. The parallel, however, stops rapidly as the concepts apply to different types of data: we are here concerned with metric measure spaces, while the measure part is absent from topological data analysis. 
\end{rmk}

\begin{rmk}
It is not required that two mm-spaces have the same total mass to compare their spectrum. The very same quantity can thus be used in both \textit{balanced} and \textit{unbalanced} contexts, using the terminology meanwhile common in optimal transport-based analysis (see e.g., \cite{chizat2015unbalanced}).
\end{rmk}

\subsubsection{First examples}
\label{sec: firstExamples}
In this first example, we used a dataset available in the package \textsf{gwDist} \cite{gwDist} presenting discrete mm-spaces of size 50 for various animal species. Starting from the triangulated fine meshes available online \cite{sumner2015mesh}, Hendrikson \cite{gwDist} obtained each mm-space by coarsening the latter meshes, following the procedure in \cite[Section~8.2]{memoli2011gromov}. From the meshes, a set of representative points or seeds is obtained and  the measure associated to a seed arises from the proportion of the mesh points closer to that seed than to any other seed. There were 10 such spaces for each type of animal. The two species (camels and lions) represented were chosen arbitrarily before seeing the results. The $\rho$-parameter was set to 0.34, the median of the distances in the distance matrix of the first Camel mm dataset. The distances are always between zero and one.
The spectral cumulative distribution functions, $\sum_{j\in \llbracket K \rrbracket} \delta_{\lambda_j}/K$, for each dataset is represented in Fig.~\ref{fig: Toy}. The lions are in green while the camels are in black. We will come back to this dataset in Section~\ref{sec: exInference}.
\begin{figure}
\label{fig: Toy}
\begin{center}
\includegraphics[width=0.55\textwidth]{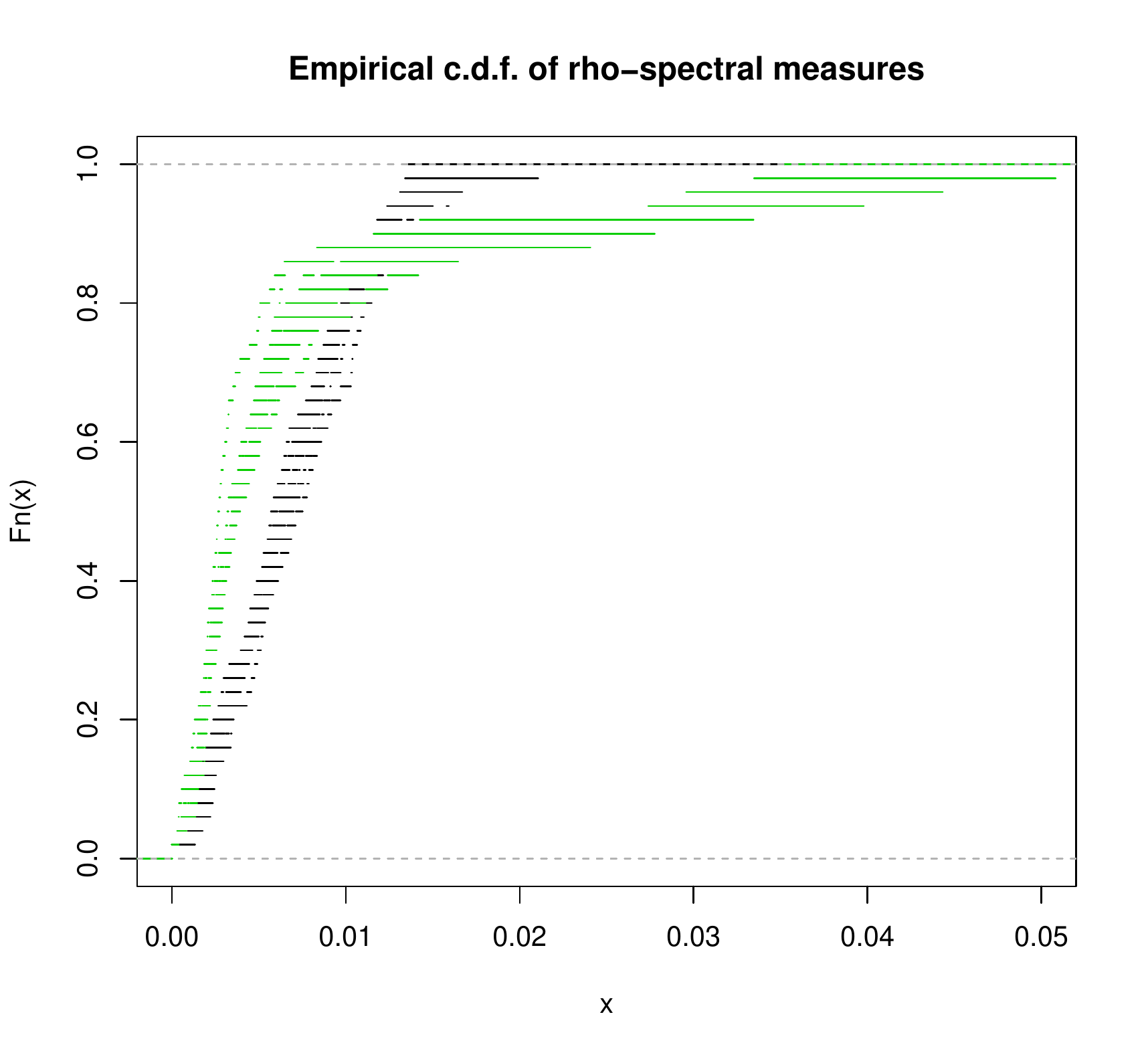}
\end{center}
\caption{Empirical c.d.f.\ of the $\rho$-spectrum of mm-spaces that encode the shape of animals. The lions are in green, while the camels are in black. The $\rho$-parameter is fixed at the median of the distances of the camel datasets. Ten replications per species are represented.  }
\end{figure}
\subsubsection{MDS on spectral distances and clustering}
In this example, we considered datasets pertaining to substructure of proteins. The dataset were taken from the CATH Database (\cite{sillitoe2021cath}) and accessed to via the function \textsf{read.pdb} of the \textsf{R}-package \textsf{bio3d}.
The different replicates were generated via random sampling of 50 elements of the $C^\alpha$ chains. The three structures belong to the Rossmann fold in the CATH hierarchy. The two first, 3V48 and 1UXO belong to the alpha beta hydrolase category, while the third, 1A2O, is a response regulator. 
For two mm-spaces with spectra $\big(\Lambda_{1,j}^\rho\big)_{j \in \llbracket K\rrbracket}, \big(\Lambda_{2,j}^\rho\big)_{j \in \llbracket K\rrbracket}$, the distance considered in this case is 
\begin{equation}
\label{eq: distLapl}
\sup_\rho \left\lVert \Lambda_{1}^\rho - \Lambda_{2}^\rho   \right \rVert_2,
\end{equation}
where $K$ is equal to 50 here and \textcolor{black}{the supremum was replaced by the maximum over a grid of size 200. To construct this grid, we first extract the lower triangle of the distance matrices and combine them into a vector. We then pick the quantiles of level $l/199, 0\le l \le 199$ of that vector to constitute the grid on which the maximum is computed.
We consider 480 replications of the sampling of 50 elements of the $C^\alpha$ chains, 120 of each type of protein. We compute the distance between each pair of samples using \eqref{eq: distLapl} and obtain a distance matrix, which we visualise in Fig.~\ref{fig: Clust} using an MDS projection. It is clear that the distance between signatures enables to distinguish between the different types of structures.} %
 
\input{scatter.tex}

\subsection{Harmonics and principal directions of mm-spaces}
\label{sec: Harmonics}
From the Dirichlet form defined in Eq.~\eqref{eq: DirichletForm}, the eigenvectors can be recovered along the eigenvalues. As they correspond to minimisers of the Dirichlet function, we call them harmonic functions.
In the case where the measure is uniform, the latter will correspond to the classical Laplacian eigenvectors for a graph whose adjacency matrix $W$ is defined by 
$
W_{j,\ell} = \ind_{\{d_{j,\ell} < \rho\}}.
$
As mentioned in Section~\ref{sec: ClasGraphLapl}, the eigenvector corresponding to the second smallest eigenvalue captures the main direction of the shape.

Let us provide a concrete example where the representation turned out useful. Motivated by a biological study pertaining to cellular dynamics, the data\footnote{Data were acquired in the biophysics lab of Sarah Köster lab, Department of Physics, University of G\"ottingen.} are fluorescence microscopy recordings of the microtubuli structure of NIH3T3 cells. Here, microtubules, which constitute important part of the cell skeleton have been recorded over time in living cells. Based on these, a graph structure is obtained by skeletonisation. The nodes of the discrete mm-space are then the nodes of the graph, the distance is the shortest-path distance on the graph and the measure of node $x_j$ is proportional to the sum of distance over all neighbours of $x_j$, i.e., $\mu_j\propto\sum_{\ell \sim j} d_{j,\ell}$.
On Fig.~\ref{fig: Harmonics}, the extracted graph structure of the microtubuli is represented (bottom). On the same figure, the top row exhibits the 5\% of the node points closest to where the eigenvector change its sign for the smallest $\rho$ such that the underlying auxiliary graphs are connected. Recall that the eigenvectors are defined only up to their sign, which explains the choice of representing that particular region instead of colouring the nodes according to the values. We clearly can identify on the left-hand side a small cluster which is only connected to the rest of the structure via a single edge that is relatively long, this isolates all the elements of the cluster as the distance is the shortest-path distance. On the graph obtained from the image 15s later, this area has a different structure and is less peculiar.  %

\begin{figure}[h!]
\begin{center}
\begin{tabular}{@{}c@{}c}
\includegraphics[width=0.445\textwidth, height=67mm]{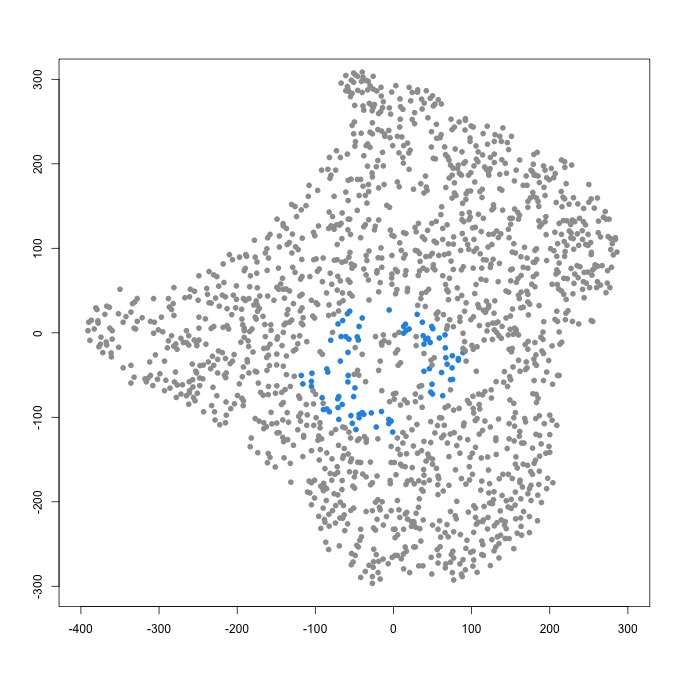}&
\includegraphics[width=0.445\textwidth, height=67mm]{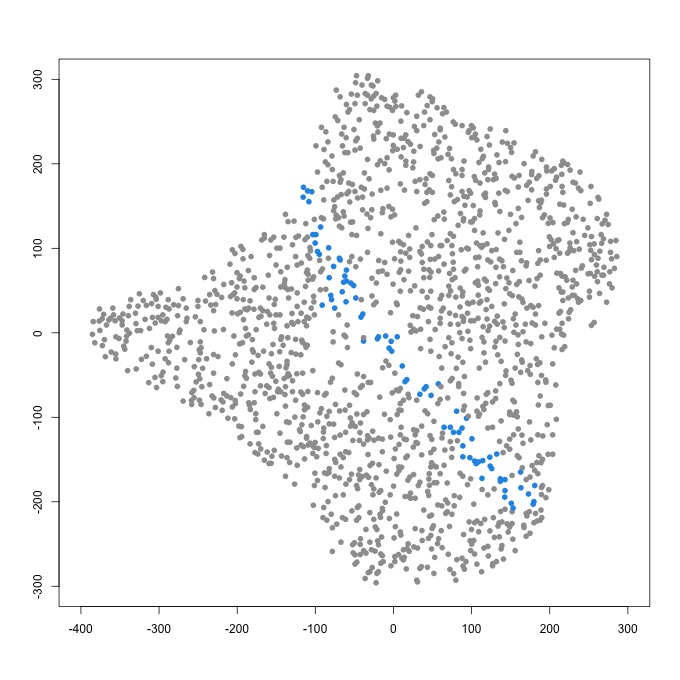}\\
\includegraphics[width=0.49\textwidth, height=70mm]{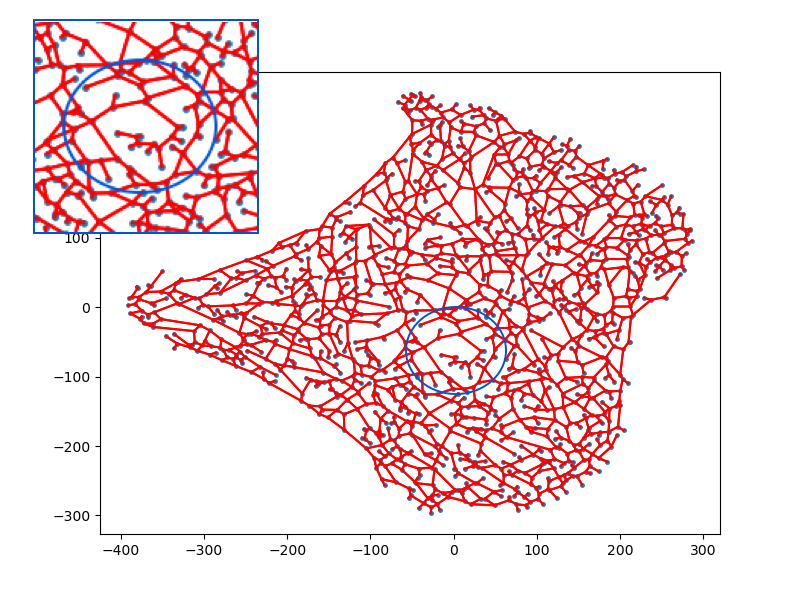}&
\includegraphics[width=0.49\textwidth, height=70mm]{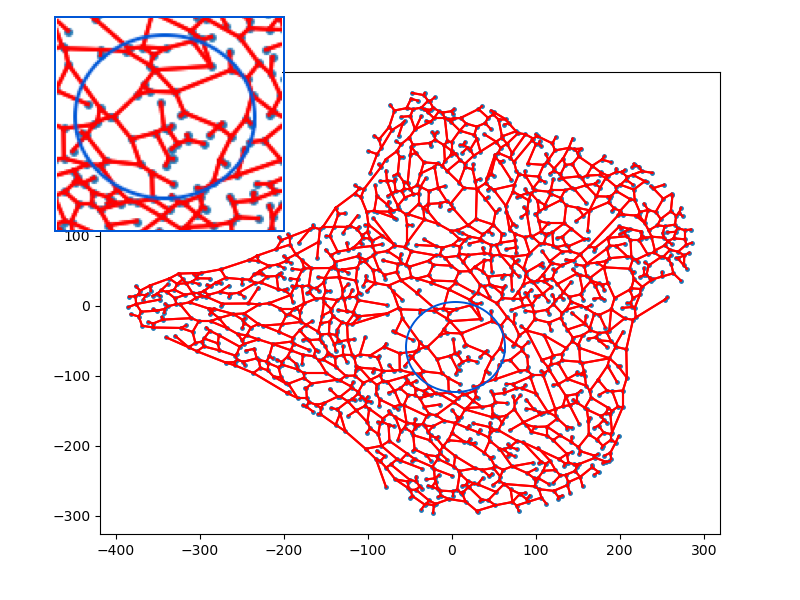}
\end{tabular}
\end{center}
\caption{\label{fig: Harmonics} \slshape\small 
\textcolor{black}{
 Comparison of temporal evolution microtubuli structures of a cell at two different time points. The right column has been recorded $\sim15s$ after the one on the left-hand side.
Top) Coordinates of the nodes of the mm-spaces extracted from the cell images.   The 5\% of the points closest to the region where the first harmonic changes its sign are coloured in blue.  Bottom) Graphs extracted from the images at these two different time points. The blue circle on both pictures reproduces the area where the first harmonic changes its sign at the first time point, as in the top-right image. A zoom in of this region is added in the upper left corners to facilitate the visual comparison. The small cluster in the left zoomed area is sufficiently large and poorly connected to the rest to drive the first harmonic. The situation is quite different on the right-hand side as the region where the first harmonic changes it sign is quite evenly splitting the shape into two parts (see the top right image) and the structure in the zoomed region is not singled out by the first harmonic.  }}
\end{figure}

\subsection{Semi-supervised learning}
\label{sec: SSL}
\bgroup

In semi-supervised learning,  the idea is to use the geometric structure included in the labelled data to infer the label information on unlabelled data. Given a set of points $X_1,\ldots,X_n$ \cite{bousquet2004measure} proposed to use the U-statistic 
 \[
U_f := \frac{1}{2n(n-1)} \sum_{i,i'} (f(X_i)- f(X_{i'}))^2 K(\lVert X_i- X_{i'} \rVert/t )
\]
 as a regulariser for semi-supervised learning problems.
 Interestingly, choosing $K(x)=\1[x < 1 ]$  and a uniform measure, this U-statistic is equivalent (up to a scaling factor) to the Dirichlet form in Eq.~\eqref{eq: DirichletForm}.  This exhibits a strong link between the Dirichlet form that we consider and regularisation in semi-supervised learning problems.

One can thus build upon this link to extend semi-supervised learning to data on mm-spaces, enabling to incorporate importance weight information.  Consider a discrete metric-measure space of $K$ points for which there exists a labelling only for a subset $S$ of $\llbracket  K\rrbracket$.  The latter set must be thought as having a small cardinality when compared to $K$. Define the labels to belong to the set $\{+1,-1\}$. The known labels are called the $Y_i$'s.
It was proposed in \cite{bousquet2004measure} to obtain the label predictions by solving 
\[
\argmin_{f \in \reals^K} \sum_{j=1}^n (Y_j - f_j)^2 + \tau  \sum_{j,\ell} W_{j,\ell} \left( \frac{f_j}{\sqrt{D_j}} -\frac{f_\ell}{\sqrt{D_\ell}}\right)^2,
\]
where $Y_j$ is set to zero whenever $j \notin S$ and $\tau>0$ is a tuning parameter.

Using the graph obtained from the Dirichlet form of the $\rho$-Laplacian, this can be generalised in our setting to the situation where 
the mass (or importance weight) of each node is now allowed to be nonconstant and is used in to predict the labels, viz. 
\[
\argmin_{f \in \reals^K} \sum_{j=1}^n (Y_j - f_j)^2 + \tau  \sum_{j,\ell} \mu_{j}\mu_{\ell} \ind_{\{d_e < \rho\}} \left( \frac{f_j}{\sqrt{D_j}} -\frac{f_\ell}{\sqrt{D_\ell}}\right)^2, \quad \tau >0.
\]

\egroup

%% file: scatter.tex
\begin{figure}
\centering
\begin{tikzpicture}
    \pgfplotsset{width=7.52cm,
        compat=1.3,
        legend style={font=\footnotesize}}
    \begin{axis}[
    title= $\rho$-Laplacian distance,
    xlabel={\small MDS1},
    ylabel={\small MDS2 },
    legend cell align=left,
    legend columns=3, 
    legend pos=north west]
    \addplot[only marks, color=black] table[row sep=\\]{
        X Y\\
-6.53851627919659	-0.119228888297249\\
-6.04127573770503	-0.669642614727853\\
-6.84771026226655	1.37958959641773\\
-5.94692070717409	-0.813585189718111\\
-5.92696616481305	-1.07769576032485\\
-6.64937009405731	0.447810282011321\\
-6.17496211805319	-0.484033205849165\\
-6.37678850340591	-0.211505852851905\\
-7.29077429046197	1.62902599157598\\
-6.29703107320486	0.513246543846285\\
-6.67977195590424	1.19474261341387\\
-6.26715865765584	-0.497575128903629\\
-6.85682722309454	0.229484005420591\\
-5.88802008934067	-1.03952648295009\\
-6.3151359812342	-0.212604572131984\\
-6.49625005967036	-0.182995639442535\\
-6.77413325692397	-0.344782595148425\\
-7.09294315295317	1.05752621032626\\
-6.10359098888653	-0.423560238204043\\
-6.81284617087772	0.584016394803293\\
-6.11280154277204	0.18312345961799\\
-6.72189255498129	0.3194923254642\\
-6.86237399862613	0.519643675484523\\
-6.75598101227685	0.699930941315769\\
-6.57844116506718	0.218675248882734\\
-6.26894176784294	-0.312513132804689\\
-7.04584521884708	0.650270064468093\\
-5.85630736303159	-0.253082966128596\\
-7.00137600465292	0.777868692257155\\
-6.78002070874698	0.483868100725195\\
-6.16262739769121	-0.117019470623704\\
-6.16755093119058	0.232002014279892\\
-6.7883612035571	0.520073957820416\\
-6.78847082965233	0.817000115168052\\
-6.83559886706562	0.54967972649675\\
-6.02470369363824	0.12404842011283\\
-6.67728586742619	0.76413310655444\\
-6.72512752454246	-0.131109330922327\\
-6.54210411339203	1.1603806634918\\
-6.12288935385182	-0.229120696120744\\
-6.39688212591588	-0.0123315059979302\\
-6.48616525384725	0.951211298387483\\
-6.327714712431	0.878439699092608\\
-7.131957621528	0.791995446357888\\
-6.37296646177325	-0.200697661531296\\
-7.0596856513503	0.70592514774744\\
-6.27988928247369	-0.148786946542669\\
-6.88481620869065	0.295008376410574\\
-5.95002238232902	-0.663803107663888\\
-6.28650240178757	-0.0850816467912824\\
-6.63761605712415	1.62569968297079\\
-6.4071678861755	-0.762048541509126\\
-6.49110767106922	-0.563600253195079\\
-6.58118263054923	0.446574378549419\\
-6.62621690397435	0.673336268443313\\
-6.47265872587537	0.223408963390432\\
-6.55932669161703	-0.115479071946035\\
-5.93698684452481	-0.694264247929823\\
-7.30568467552164	1.40834533925811\\
-6.22607740012322	-0.260142959318191\\
-6.93898997983596	0.270929234529446\\
-6.66125160520926	-0.242342740681806\\
-6.36920664443598	0.209772123241055\\
-6.43206413989321	-0.235304257494688\\
-6.72907216836839	-0.016143245921054\\
-6.76026052250563	0.475679141337976\\
-7.2163191589471	0.831558529743557\\
-6.94817859194034	0.686378517724103\\
-6.01649347459004	-0.653000055787896\\
-6.50413164226282	-0.18760339252382\\
-7.37211416093964	1.89549368691496\\
-6.95462259266993	0.458763910114019\\
-5.71220664712913	-0.995612398536764\\
-7.5997865007919	1.72854671521924\\
-6.7499099816016	-0.0867555568027914\\
-6.47697135262471	0.19877018217764\\
-6.49994219771448	1.01174221537304\\
-6.66608961858349	0.681466318348445\\
-6.71424005618079	0.798856344938583\\
-6.32961183508299	-0.310857841328187\\
-6.31294395912059	-0.517931338972952\\
-6.77303602302938	2.01344653584713\\
-6.85095385246578	0.864149312518933\\
-6.54205848222724	0.794119008551075\\
-6.08129260993647	-0.945558672078698\\
-6.502632110133	0.143165313015784\\
-6.83486612909734	0.275392089447168\\
-6.48056746564287	-0.0517457643335266\\
-6.3353624174936	0.127116922026865\\
-6.61058042309988	1.16368544447793\\
-6.99603096765718	1.49861375892964\\
-6.35601476591863	-0.652015640864633\\
-6.22548112042681	0.699352720850627\\
-6.88747461743958	0.977626763340096\\
-6.74520028130363	-0.111372205422856\\
-6.86843731369217	1.18158634734419\\
-6.06498838217633	-0.554971921954787\\
-6.23426336449256	-0.154922264608575\\
-6.32315036564317	-0.142067418811756\\
-6.32680776143508	0.198196335969929\\
-6.40158845202535	-0.265776081087418\\
-6.85313725068994	0.709500241092248\\
-6.49895764775217	-0.4152736006603\\
-6.29245679439929	0.029915605142414\\
-6.2629540113061	-0.363328865516\\
-6.86381678501729	0.869918306312584\\
-6.51162264392144	-0.27146550929112\\
-6.68699913218943	-0.0320862091448028\\
-6.54511269735102	-0.110636261198412\\
-6.22841912695282	-0.270611979279444\\
-6.51373372517105	-0.13472726229006\\
-6.41499256991215	0.489228467875659\\
-6.62811372331125	0.377836248845811\\
-6.37006597428619	-0.139079860659374\\
-6.67740857573746	0.890014672105006\\
-6.47911689738837	0.746363738858467\\
-6.36671866944886	-0.201500000424502\\
-6.60464233304515	0.933042004499761\\
-5.33046686238081	-1.59438641215168\\
-6.2225998292248	0.111545225721468\\
-6.43549654723184	-0.0381426932901384\\
-6.38558156404816	-0.398957913222539\\
-6.32928962234509	-0.184478389887986\\
-6.70856701878838	0.185485317910639\\
-7.36655693207242	1.52153301637646\\
-6.36821234078976	0.344419710243399\\
-6.68763629001564	0.391078820936249\\
-6.57955500107494	0.538910419357903\\
-6.17330447196919	-0.915723647864825\\
-6.33895753269728	-0.0446477613816682\\
-6.91538020605743	0.108240846578057\\
-6.3651138724386	-0.276180355445898\\
-5.7836087818019	-0.578887832292425\\
-6.31542734335635	0.191958207627834\\
-6.57336675430901	0.283175989450389\\
-6.30431530941542	-0.216949759709215\\
-6.49311408548068	0.326380880135917\\
-6.64061260330637	1.18055957034673\\
-6.50103990021019	0.400390072864809\\
-6.43362917593645	-0.198841301732995\\
-5.98854370022004	-0.618948012928957\\
-6.57907963216093	0.130019019355016\\
-6.12657446059521	-0.71762601665326\\
-6.42536913423359	-0.328778240649055\\
-6.25382066360422	0.0763065942578405\\
-6.5838999287237	-0.221942916259895\\
-6.48595851578684	0.238375974348059\\
-6.00904950024133	-0.489970128241839\\
-7.35031612915642	1.07536777203546\\
-6.2516495961921	-0.428565336626127\\
-6.2856962969075	0.123238000249701\\
-7.10913814308476	0.785030210574309\\
-6.99195963346645	0.317352330760089\\
-6.31498242008734	-0.0407588051129758\\
-6.85682289021535	1.3616981232636\\
-6.72702626503226	0.521712109852937\\
-6.1004981217469	-0.865466179948288\\
-7.01661169115639	0.269891374817773\\
-6.01362123428015	-0.480030787091461\\
-6.38036975887444	-0.625565455542672\\
    };
      \addlegendentry{ 1UXO}
    \addplot[only marks, color=red] table[row sep=\\]{
        X Y\\
2.69626359082641	0.311857298291348\\
2.9312369616863	-0.798691385916119\\
2.09597441543051	-1.26635842575053\\
1.35910744801571	-1.74901648272675\\
2.57980307828437	-0.720841216568426\\
2.84924448138521	-0.660392064784167\\
0.574404935002526	-1.7007025696592\\
1.386212254005	-1.22235837246675\\
1.6687532958272	-0.323681932267062\\
2.46696187089847	-0.60120438608642\\
1.88220434419245	-1.29549934605232\\
2.77488655820762	-0.837139274057453\\
3.12555507171675	-0.773828942201475\\
2.49507056096685	-0.132798154954581\\
2.36773729311799	-0.544395607351274\\
1.59156957174786	-1.11417977216092\\
2.37337552258587	-0.71107616808178\\
2.89151092040046	-0.974618021731261\\
1.90511244215028	-1.1510300474836\\
2.38301771729045	-0.926930268259639\\
2.20856096812083	-1.09042978695614\\
2.77164550225706	-0.658503119416596\\
2.29804740524816	-0.232382012194563\\
2.4692912882239	-1.69745696929886\\
2.62639970331908	-1.68379753576898\\
2.47358317655582	0.0268788402808845\\
2.42644035549553	-0.254420106172573\\
2.32414099017287	-0.164799595959359\\
1.92867813218001	-0.935659490871567\\
2.24849009130127	-1.23444429955797\\
1.50291442758032	-0.751140055304195\\
2.18345461921307	-1.21096779715487\\
-0.0166341168191242	-2.13615092837834\\
0.226985061527658	-2.07438181457224\\
3.56378255733426	-0.104649720460675\\
3.09121715509721	-0.55338901053054\\
1.3041912537395	-1.41420193743781\\
3.15204732901564	-0.769924207314451\\
2.18736759763237	-0.819699921544644\\
2.69466861322576	-0.329503306464009\\
1.28555552728968	-1.42260643526649\\
1.14432845520842	-1.29770809625916\\
2.98189371564822	-0.948487055826508\\
1.7748696169128	-1.0076759674568\\
2.79715344623814	-1.17671554129493\\
3.06058836620479	0.155581970987073\\
2.86882134826216	0.507582199457699\\
1.82075194486713	-1.24107576679378\\
0.578219745524844	-1.57603583031538\\
2.25128782928539	-0.77038925447525\\
0.971894865122474	-0.801398978648618\\
2.08302177125228	-0.723770907278927\\
2.20726401857495	-0.815574204888612\\
3.02636135261664	-0.347247675573938\\
2.96088970953927	-0.560953425438414\\
1.45609662849238	-1.69017064997472\\
1.69073413252472	-1.06371472232735\\
1.56324913138867	-0.986565968582125\\
3.00333515263454	-0.520036052180679\\
2.11800595857953	-1.50510270774486\\
1.01551350296206	-0.787988764709333\\
2.12454050314927	0.112066173153436\\
2.27358772329118	-0.926531124168575\\
2.65967487380026	-0.433120143623047\\
2.01283245241811	-1.09574847569193\\
2.46187828395772	-0.463409121067926\\
0.585375930281713	-1.50759631002449\\
1.34386521738495	-1.16737178910134\\
2.70964465528441	-1.42133300659719\\
2.78562194238202	-0.290239052345175\\
1.07057869119026	-2.06311148591985\\
1.95916214567602	-1.43531112872184\\
2.62894935756579	-0.727609756982023\\
2.97587026453362	-0.671179125323845\\
2.06742604894874	-0.603588883451025\\
1.8159047853569	-0.326566978504686\\
2.07042218599496	0.229508820327161\\
1.80184526832519	-1.03617610228056\\
2.49659781938499	-1.00121305931269\\
2.18584005634836	-1.32131803906635\\
1.80646887712949	-1.20134607971241\\
2.03992284372576	-0.759581433021968\\
2.2469357383129	-0.582726159115063\\
1.90604207097664	-0.271322828431841\\
2.54331859557903	-0.662520126236754\\
2.25620083596487	-0.508220382712714\\
1.96987300347326	-1.24014645773684\\
2.81695060855217	-0.118554116163112\\
1.68978252716633	-0.805549225328628\\
2.47804851388821	-0.723974042285974\\
2.51520453939802	-0.896424814752581\\
1.67346588301919	-1.28694486595393\\
2.70391074072116	-1.43785891129671\\
2.85856874585545	-0.147431438422471\\
2.38598675627864	-1.22054558855287\\
2.65656296000702	-0.163677994072878\\
2.01003506276619	-0.226028599780608\\
2.47239367507584	-0.967423955253649\\
1.86282659297647	-0.937690655768114\\
2.17128009535515	-0.824152726871002\\
2.09946472016832	-1.20034063173488\\
1.66431064210771	-1.52641488874521\\
2.48909659945683	-0.690955793909245\\
2.6651544522791	-1.55887875902541\\
2.63560007563038	-1.4519852631312\\
2.20920452100083	-1.67161136981296\\
1.60519525033016	-0.773004989519768\\
2.0643403159516	-0.613207908254865\\
3.39656974862046	1.27885903369183\\
2.37595017023042	-0.353026914138508\\
2.5564491498546	0.254383140010827\\
1.53367513557207	-1.10523144309269\\
1.48050629629092	-0.0795059118098419\\
2.26862727072886	-0.488902384206639\\
3.23271245408714	-0.585213880332885\\
2.50357117784708	-1.78432832715418\\
1.4385971734265	-1.03892112955372\\
1.60727425747293	-0.743125780242683\\
1.72419386522773	-0.63624888364048\\
2.17478126970217	-0.596267290686194\\
2.22826822532586	-1.59425958674653\\
2.10940589708434	-0.638796143033604\\
2.31031317532649	-1.21032385448937\\
1.97748001165159	-0.17162076318764\\
2.55068879232077	-1.54499922552152\\
2.40116811988172	-0.733643083455683\\
1.66717407744717	-0.763334973975358\\
2.46997850608618	-1.86912659009026\\
1.26231353764309	-1.38839576232774\\
3.09714823424891	-1.44679809577844\\
2.46170437946156	-0.658081343086918\\
3.17568152930568	1.07492656961601\\
2.54081908235695	-1.09394224589076\\
2.79535393128332	-0.679105547710764\\
2.41413832533399	-1.15674233649115\\
2.03272529582536	0.145354239661425\\
2.25776773620991	-0.733404597938859\\
1.2780957413102	-1.13908183062412\\
1.74616248084279	0.0686614807405787\\
1.0751319151227	-1.59072594708358\\
1.88362363426975	-0.360502006962959\\
2.18139039426806	-1.11416825798036\\
2.79575703491605	-0.54854897903209\\
2.41041105151116	-1.18438833711468\\
0.877990718199828	-1.26500004613505\\
2.39805235977553	-1.02805688502385\\
0.805211601757453	-1.89305525232324\\
2.10810834634942	-1.03083185450709\\
2.60677510566017	0.668380020515388\\
1.90770763742077	0.356966859114003\\
2.75394618828618	0.639598321161169\\
2.55909242524182	-0.24291291131745\\
1.57808955916704	-1.06444768220666\\
1.94471725886614	-0.196658610545346\\
2.35741768956615	0.384412799169802\\
1.09635592870685	-1.18298763220149\\
2.03259747920654	-0.91880875516123\\
1.76138409721421	-0.69025936580378\\
2.55644462835777	-0.482236517076474\\
1.95377630157902	-0.199541411449057\\
        };
         \addlegendentry{ 3V48}
            \addplot[only marks, color=greenalt] table[row sep=\\]{
        X Y\\
4.12735421279728	0.437499973422226\\
4.30819489809721	0.575887049864269\\
3.98896273234811	-0.51704538226745\\
4.17045288859786	0.550465059658137\\
4.66304664570228	0.249356581678216\\
4.50507021937739	0.309310747435216\\
4.71861976520428	0.24313122159242\\
4.84956726826664	1.502366852615\\
4.80606803612697	2.04872129610279\\
3.84294861811025	0.202007401189687\\
4.67939091692757	0.765695240892771\\
4.22324499928862	1.22537545057207\\
4.3108618684023	0.598436641045592\\
5.14597420779002	0.767950699557014\\
3.31925679293643	-0.528710271403386\\
4.30053252603037	0.737091829760099\\
5.10965416269145	0.697089600750845\\
4.52720811958524	1.3433657861526\\
4.55718872266661	0.715648925191586\\
4.67610060899608	0.530845065911749\\
3.99823158960795	0.571946215525166\\
3.85626984703877	-0.696199508090838\\
3.97416445051722	-0.519233401606659\\
3.90535456866897	0.0574148905662352\\
4.49488197572902	0.362291893352694\\
4.0597132248745	0.64447525417668\\
5.13883574538536	2.02738639717108\\
5.52939075823249	0.57542909632254\\
3.87958649604768	-0.27266875856324\\
4.16243023106831	0.0754903759412377\\
5.13096774331105	1.67343230275487\\
4.29224087402863	0.621919907475599\\
4.23650598468865	1.55465523489036\\
5.02300676101754	0.313830010928474\\
4.55368766775177	0.615687056660911\\
3.845809030223	-0.0271573186582053\\
5.06693098699697	1.40855020143258\\
3.90172127328048	-0.777070251354718\\
4.38171925374827	1.0552988900289\\
5.13374137938477	-0.0606719271229548\\
4.7274548392279	1.03664590134282\\
4.44747631660906	0.781384348955928\\
3.21678189221744	-0.634924726000707\\
2.59548473317748	-1.14969888424723\\
4.31945247632526	0.70950172806719\\
4.24248013119172	0.13160531170189\\
3.69835339003017	0.073840779862609\\
3.68930253023608	-0.0836091796892283\\
4.44000725313242	0.882892285908623\\
4.68327858345147	1.63685472655441\\
4.91599964302449	0.700690264372361\\
3.69367595480819	-0.974990721327056\\
4.3328202024852	0.729278468179255\\
4.15649781028041	0.579617267089052\\
4.48102998597495	1.01135586256325\\
4.37983448199026	0.509886322566622\\
4.58544993678245	0.604866448033834\\
3.39433541929233	0.735061489460557\\
3.68038231250211	-0.0506089607836383\\
4.36214764842683	0.599409641035335\\
4.65583002159148	0.48584743749138\\
5.07494688196268	1.8645298163165\\
3.82789541580645	-0.259332191972938\\
4.08635574743367	0.0211396600728658\\
3.69604881183004	-0.0264715696852868\\
4.49134814912552	0.404956321116573\\
3.90012581119723	1.47228496111237\\
4.49021962466967	0.454915759123332\\
4.42892917534883	1.1702557544885\\
4.98206662525037	0.682859612462036\\
4.13602068727249	0.835212531470849\\
4.99880129348827	0.385142911345877\\
4.00226293855915	-0.355792557360265\\
4.95311495933783	1.22955201228952\\
4.84746269184938	1.36080906908531\\
5.24900182547599	1.47928809784761\\
3.98891009803205	0.132239060431583\\
4.23951141794747	1.16419461604378\\
4.07247820731375	0.673147778112815\\
3.87663544326218	0.422089984302843\\
4.09153199531028	0.784682626373863\\
4.18427472752055	0.733201822181341\\
3.4436376514451	-0.405466921807779\\
4.39424365961199	1.06229320964129\\
3.44904412211019	-0.0314745489054613\\
4.18059082683427	1.15028667235629\\
4.3557181120998	-0.445808734779859\\
4.41636680329554	0.614968317433502\\
4.20127161553982	-0.56464865724916\\
4.28291346792988	0.97815538096054\\
4.21672309779968	3.0223218356532\\
5.14594238942402	0.888272692994024\\
4.58268576947421	0.289362036677105\\
4.30901778959557	0.404771062348157\\
4.57505743732478	0.490080420942343\\
4.35592617693777	0.880458679691313\\
3.75960167993828	0.296627800902639\\
4.19478923173698	0.801620073173867\\
5.22956218964287	2.61058719485532\\
3.34508188815612	-0.286521434451745\\
4.14045386266815	0.885463992874075\\
4.52754044315767	1.11806335804795\\
4.73640227482784	1.24827990882741\\
4.97243122056773	1.37050091497523\\
4.83186697965912	0.95038794168395\\
3.99191215944316	0.930534342887019\\
4.63787555311187	1.0102934242273\\
3.56668812946138	0.529566590718593\\
4.65774711348191	1.11285210796583\\
3.23080525663844	0.360520676965711\\
3.91064016008516	0.074412124932937\\
2.47789472318189	-1.67351068288927\\
4.10091203235379	1.2047058658545\\
4.86558357164592	1.59348641416381\\
3.98464310495938	-0.0230513403511414\\
3.87494677934921	-0.49101908879153\\
5.17317556460113	0.527407945298691\\
4.99038704111495	1.14444575481129\\
5.01870188164861	1.35997524365655\\
4.51561851236295	0.390745381692582\\
4.88193527338459	0.821777267491085\\
4.9477820282445	0.638022762846833\\
4.81633615930055	0.85144764337662\\
3.94539602408818	0.76123404724094\\
4.54929421426063	1.05072085160331\\
4.79492389340064	1.30841237422714\\
4.38669003760437	1.0288956830837\\
4.89126698474146	1.2691907756869\\
4.7751940870926	1.2080445443336\\
4.16936205822281	-0.0394765702868703\\
3.51391969502858	0.141739604450424\\
4.28455904532476	0.00471138864988384\\
4.94871436760541	0.623542576525216\\
4.66469067141899	0.938186371940083\\
4.65687577618079	1.39698348592613\\
4.59418632843115	0.726103589604336\\
4.79918757713835	1.15595570403814\\
4.32224646735973	-0.0228760255567547\\
4.23198194674583	0.416028961688546\\
4.22235703813827	0.386430528808908\\
5.11415812581214	2.42482618403365\\
3.52960672284922	-0.537019965987229\\
4.45670575461168	0.822680706010205\\
4.98246905180338	1.26734525389614\\
4.00988359468713	0.771415588289085\\
5.09346378208569	0.851147482644245\\
4.7643723609406	0.503769889053802\\
3.73399662733824	-0.439190537610523\\
4.54936743781154	1.07509239701452\\
5.04110906392388	1.78768565014628\\
4.77960833914672	1.56673651055929\\
4.67509556700374	0.60180461619657\\
5.47807012178901	0.672293843198584\\
3.79148517548711	0.0753654728909284\\
4.23407306394262	0.954690994544148\\
3.66848354814732	-0.272259592068986\\
5.11951277929497	1.20312009209647\\
4.80430181870167	0.832518332773269\\
4.04272457123963	-0.183223973075821\\
4.33975941777343	0.392300476517845\\
                };
 \addlegendentry{ 1A2O}
\end{axis}
\end{tikzpicture}
\begin{tikzpicture}
    \pgfplotsset{width=7.52cm,
        compat=1.3,
        legend style={font=\footnotesize}}
    \begin{axis}[
     title= DoD,
    xlabel={\small MDS1},
    ylabel={\small MDS2 },
    legend cell align=left,
    legend pos=north west]
    \addplot[only marks, color=black] table[row sep=\\]{
        X Y\\
-11.6627866429456 	0.240522730726944\\
-10.9323622089883 	0.5188213028894\\
-12.596790921446 	-0.184702193015231\\
-10.8911836707383 	0.299931274294208\\
-10.7565879435773 	0.556394411656934\\
-11.7596261528285 	0.489156001153512\\
-11.0466481976502 	0.52835529197016\\
-11.5297447869746 	0.366951813727424\\
-13.1441346889354 	-0.0290403015942811\\
-11.5099576334916 	0.0742337950733127\\
-12.5352042911888 	-0.243965722469198\\
-11.1084071193851 	0.529257591853089\\
-12.1356058881349 	0.238046579998878\\
-10.5774365628211 	0.756211352739979\\
-11.2748399845338 	0.483584519904271\\
-11.8198250964304 	0.0541063238056074\\
-11.758168476846 	0.443046000830801\\
-12.907040684343 	-0.103369154489198\\
-11.1280151230084 	0.401396175047196\\
-12.1305304129495 	0.0119634694594448\\
-11.3208534103515 	0.163801654714128\\
-11.9852888417004 	0.281621701515565\\
-12.3802654614778 	-0.0226645169421229\\
-12.3817315056248 	-0.0427822668661013\\
-11.832243385676 	0.179284896152306\\
-11.1840972429988 	0.579926250397704\\
-12.4529144293293 	0.242585495003652\\
-10.9418979502154 	0.176097334746868\\
-12.5705290891437 	0.0567962082004127\\
-12.1226680828797 	0.113780939427591\\
-11.3882223824081 	0.0487736176238603\\
-11.3541444477536 	-0.00733812342022658\\
-12.1970717590475 	0.0787544659488661\\
-12.3368654934534 	-0.0154237641782783\\
-12.1475582786185 	0.204462824866\\
-11.2980375047966 	-0.0178218730273919\\
-12.4961235934886 	-0.193540285292002\\
-11.6476228700307 	0.758314388713707\\
-12.3325314660472 	-0.327668549074702\\
-11.256728497625 	0.26609938865329\\
-11.7450850056259 	0.0676720646283053\\
-12.1113431478751 	-0.284614083193362\\
-11.7561046737991 	-0.173663452438568\\
-12.5051026536289 	0.184190701444638\\
-11.5037580141073 	0.24980586644202\\
-12.3223540648733 	0.416907367900318\\
-11.574384254743 	0.0450213536493313\\
-12.3189525803774 	0.174073822984486\\
-10.683008612056 	0.645010815774367\\
-11.261218482472 	0.427850878233334\\
-12.1556136914427 	-0.118848294910537\\
-11.1556113808605 	0.71898486981642\\
-11.3448700389315 	0.6393411263951\\
-12.0373006426753 	-0.0289201001026995\\
-12.1899589907832 	-0.224073418395202\\
-11.6824232735918 	0.196990146616534\\
-11.7512523611208 	0.247765714925964\\
-11.207708724491 	-0.00562850168541491\\
-12.8335180272192 	0.112291172187199\\
-11.1450170481562 	0.522716419981606\\
-12.2756403006543 	0.284266965676632\\
-11.8664270844678 	0.260815665495221\\
-11.5633509063204 	0.20670187719797\\
-11.4203154643676 	0.592235473988363\\
-11.7447477237453 	0.454487113031995\\
-11.9058370641905 	0.339866311247715\\
-12.8183884233377 	0.134489407557878\\
-12.129321203119 	0.411399558348795\\
-11.1464443040195 	0.187598925549189\\
-11.6571271894736 	0.149077389845746\\
-13.1789218646716 	-0.0454752662894118\\
-12.2785572964664 	0.339935420953997\\
-10.3765107514687 	0.583472744870002\\
-13.9260412099879 	-0.265069292671016\\
-11.7544515781865 	0.542086486002165\\
-11.8369098931493 	-0.0160741339179535\\
-12.3040025306023 	-0.411304077678929\\
-12.1270544631008 	0.138698180125877\\
-12.0393368016967 	0.0829318937779055\\
-11.1693247980665 	0.57024320199702\\
-11.3713178004494 	0.42250770336308\\
-12.6031687315905 	-0.36871995192471\\
-12.3033977082912 	0.083045242455548\\
-12.1832564094812 	-0.153802546344728\\
-10.8389194838661 	0.702812021637835\\
-11.6787117754725 	0.272126120723977\\
-11.9940054188441 	0.386517402336043\\
-11.5190602677188 	0.358215415550971\\
-11.5172142489917 	0.272099507216966\\
-12.443175463578 	-0.287703365822764\\
-13.2123793541812 	-0.540268522300787\\
-11.5438290216327 	0.271159929763603\\
-11.5183527202248 	0.0373873231507875\\
-12.1261570798361 	0.323116714001443\\
-11.7632304152013 	0.581489103551167\\
-12.4882649597207 	-0.156198407139424\\
-10.983800991869 	0.474512254609466\\
-11.286111969285 	0.309686739165163\\
-11.4155796784479 	0.367092128063678\\
-11.5713539144262 	0.0443277371176276\\
-11.5346351316563 	0.277834978569471\\
-12.5294657259802 	-0.0382874567257582\\
-11.4997255631234 	0.528398102315356\\
-11.4252448980285 	0.169618502938728\\
-11.2226620961719 	0.489285142128176\\
-12.6176705953217 	-0.13058655890883\\
-11.6686130055424 	0.293446307806816\\
-12.0308348010639 	0.295269897782878\\
-11.4374513899596 	0.539937216429238\\
-11.3643035948665 	0.275161248679153\\
-11.8423154140175 	0.203909802303513\\
-11.7448000050296 	0.0733921591247113\\
-11.8197508260684 	0.264351841490999\\
-11.3385236751086 	0.481552434507369\\
-12.2583019072638 	-0.0107131444006823\\
-11.7421438371187 	0.131579198379115\\
-11.4096801405218 	0.197795796671107\\
-12.13344368782 	-0.156919049219164\\
-9.91895026019568 	0.677572558442372\\
-11.4367436627037 	0.109743391435849\\
-11.5358951744337 	0.187548250992094\\
-11.5064577435034 	0.213832547497497\\
-11.2400180648451 	0.345235490370473\\
-11.9573192530989 	0.195743319741112\\
-13.1713041279669 	0.0127627421064537\\
-11.8119888630377 	-0.0668822375027812\\
-11.7434379329526 	0.438931952889026\\
-12.0870504438013 	-0.0602876933390392\\
-10.9719891195995 	0.582119673523366\\
-11.6437791562191 	0.0722376120956146\\
-12.1387907922905 	0.496669973338213\\
-11.5381370249922 	0.223008213862919\\
-10.7965659918796 	0.319674693337699\\
-11.5835982571854 	0.114419512243512\\
-11.8358206599077 	0.137032584582539\\
-11.6283290618168 	0.0538674265811097\\
-11.5884827578474 	0.424362922700143\\
-12.6220521921338 	-0.406754237115728\\
-11.4529883443563 	0.40095451571018\\
-11.4156738065297 	0.386276519552989\\
-11.063127307843 	0.346404043720602\\
-11.6737718055451 	0.209982382101807\\
-10.9654968208152 	0.67263487301137\\
-11.4425573944439 	0.494846688590389\\
-11.5468141015742 	-0.000675455376260796\\
-11.8941730226932 	0.246579662098555\\
-11.6045224210422 	0.198982992330306\\
-11.0847575612479 	0.249575740616804\\
-13.0211284044494 	0.0818496372407309\\
-11.2908939332705 	0.219296145062614\\
-11.5567683298444 	0.111936061075641\\
-12.6667096268199 	0.134045494376241\\
-11.907601874112 	0.536726367687558\\
-11.5106093014609 	0.211055367212853\\
-12.5397108817653 	-0.0257142579777974\\
-12.0044407658136 	0.113610261372468\\
-10.8520858872978 	0.854523117461829\\
-12.2323443800829 	0.331452396772765\\
-10.9054068230632 	0.320528986489904\\
-11.2962005305391 	0.504735878551813\\
   };
    \addplot[only marks, color=red] table[row sep=\\]{
        X Y\\
5.86615231420553 	-0.970451457249128\\
4.43102574119913 	-0.265273817527278\\
2.45889702358476 	-0.363332134967097\\
0.954638331397603 	-0.51154217887624\\
4.94037424069285 	-1.30208711281956\\
5.32042805417574 	-0.928758049079782\\
0.375218494611852 	-1.45811910093393\\
1.69173784629408 	-0.855706550352013\\
3.56079379396132 	-1.66159852850396\\
4.23640503453039 	-0.681202131597204\\
2.18980757662194 	-0.303659514448338\\
4.416169077013 	-0.54452782179141\\
5.12413656743044 	-0.289956152133347\\
5.34491703322677 	-1.07041922743014\\
4.38431281971945 	-1.11034020305213\\
2.20691555376477 	-0.813534528799064\\
3.80853415028907 	-0.547150091296382\\
5.07590620516831 	-0.84992324234232\\
2.00581683287648 	-0.250350316527358\\
4.26418415411813 	-0.978037766371831\\
3.6341527459376 	-0.858963819235731\\
4.05198101682439 	-0.112087993981358\\
3.80935254608855 	-0.776520299350056\\
3.10929983762954 	-0.465726791738785\\
4.25595625707569 	-1.21900270250111\\
5.31726048619523 	-1.05341405337099\\
4.75752637695437 	-1.13990841069711\\
4.10429633283233 	-0.855564345423669\\
3.57331522109818 	-1.36607109560891\\
3.38907246387467 	-0.852414821102385\\
1.91657164743071 	-0.81081893159179\\
3.09197827577619 	-0.65726040517261\\
-1.19323836621715 	-0.907076102268378\\
-0.994640603088578 	-0.745240840670393\\
7.23687526012998 	-0.55685326192622\\
4.84075145689121 	-0.251543287122701\\
1.18574408439061 	-0.725785704209625\\
5.09052968592476 	-0.376619246902839\\
2.75880025361122 	-0.294324119336829\\
4.49377631656894 	-0.393136384220236\\
1.28603641937467 	-0.896297225611946\\
1.20163057571668 	-0.735943217614625\\
5.37197520586626 	-0.661553428625377\\
3.25419393139836 	-1.25909766705779\\
4.7637081986228 	-0.906657663937102\\
5.57324789908419 	-0.113554646710792\\
6.77267393935024 	-1.02038796207555\\
2.4382575192216 	-0.646911926247982\\
-0.578038558564548 	-0.522185585178378\\
2.80730190883964 	-0.106330947922893\\
1.29269801011754 	-1.38072882800091\\
3.15648864287265 	-0.542410638258228\\
3.18948935265194 	-0.52780241549156\\
5.93462581841728 	-0.883247629936456\\
5.53734943548645 	-0.704730641709677\\
1.40235602562137 	-0.756258879055709\\
2.61887116977648 	-0.999915209757664\\
2.15549952540743 	-1.00189333000855\\
5.27458927359726 	-0.303194506947436\\
2.56526381879196 	-0.446548799530551\\
1.89910651500235 	-1.75562714969998\\
4.01141585865987 	-0.726702409466578\\
3.88070241638747 	-1.08804812777553\\
5.1290718361034 	-0.901584658629398\\
3.23847189245302 	-1.16748565837056\\
5.04036531115968 	-1.20778053152319\\
-0.0956197111402921 	-0.938896657357067\\
2.17230185696023 	-1.53285220547298\\
4.13392747299513 	-0.764486995576755\\
4.6422715838379 	-0.285368829201302\\
0.561073687604614 	-0.714316102601369\\
2.81709486583368 	-1.14842818878765\\
4.21311685579551 	-0.424484065238224\\
5.11118333704532 	-0.52881079652446\\
3.197591878114 	-0.597457780376654\\
2.15814624691371 	-0.280332695491629\\
3.89460166591787 	-0.905792841287388\\
2.85190519406757 	-1.13057895630303\\
3.89241449030059 	-0.967288752324994\\
2.61274244655433 	0.0455262515642288\\
2.42134060438734 	-0.896027192095701\\
3.11196670942963 	-0.829870148741829\\
4.77875302289083 	-1.55241277615687\\
2.60751503537392 	-0.296131987797507\\
4.53300267145652 	-0.748579535710364\\
3.8073971019388 	-0.836396677358663\\
2.97564482229978 	-1.16941149167816\\
5.80495961879786 	-1.263114483173\\
2.62309021944712 	-0.985931892789172\\
4.08019816556187 	-0.749142457983045\\
4.4996965556204 	-0.602351005487591\\
1.6078190669494 	-0.541680481081362\\
3.56188855959723 	-0.337958680062625\\
5.45471102534328 	-0.731631177055323\\
3.59294067896174 	-0.657820237502202\\
5.55584535118527 	-1.23825275856766\\
3.27388269204639 	-0.825721710009633\\
3.74223945346293 	-0.510444640896611\\
2.84431239780455 	-1.0608711455399\\
3.52575604285281 	-1.06988573499238\\
4.16552588504551 	-1.88969753822225\\
1.73471991420857 	-0.479213958665205\\
4.20749653269066 	-0.696960224568146\\
3.47115594881611 	-0.4146828735571\\
4.19245783791524 	-1.14728301139959\\
2.82496720318406 	-0.652983869758834\\
2.45719740222595 	-1.08908698071714\\
3.99307968352735 	-1.71354768177701\\
9.02512090084545 	-1.23626366556451\\
4.18905224461841 	-0.694909825632347\\
5.44629070069721 	-1.03749199599985\\
2.08427584309989 	-0.914005096529958\\
2.54881096788251 	-1.32771607794909\\
4.81728016573093 	-1.39956320958837\\
5.79110640696715 	-0.572111937970525\\
3.45139025691844 	-0.833450624333279\\
1.61187296199143 	-0.705859020281918\\
2.92608788411029 	-1.5923694602238\\
2.9774477949793 	-1.17323368686103\\
3.23031038588135 	-0.415269934125239\\
2.83878807747161 	-0.505059700542928\\
3.24337828477914 	-0.588035524441686\\
2.7093079091659 	-0.12438174300495\\
3.94800711260985 	-1.17913904572259\\
3.41779708940328 	-0.598928390516802\\
4.38865587658295 	-1.22013274177236\\
3.40789989121625 	-1.74160295458488\\
3.49166733130515 	-1.17293431080413\\
1.13764760041964 	-0.880627753396257\\
4.08738280670824 	-0.0800503902076267\\
3.52831995153706 	-0.243349156855164\\
7.78928395088343 	-0.824418181445391\\
3.78622600996528 	-0.586713903063533\\
4.97594107390673 	-0.562389805772631\\
4.18186191864934 	-1.29061825600958\\
4.33097556169604 	-1.41259182061652\\
3.67731967965421 	-0.851162405061383\\
1.68177582841272 	-1.32174210759065\\
3.94548184837304 	-1.74654812778877\\
0.952321652240769 	-0.873771809664037\\
2.10713304611872 	0.059277980832677\\
3.66305775887569 	-0.959850086220593\\
5.42391463836511 	-0.997502958025539\\
3.66565420919635 	-0.666267448615088\\
0.710227040359223 	-1.0181110477206\\
3.70952087826795 	-0.594154133482499\\
-0.262622686011616 	-0.350815891998643\\
2.15068117385956 	0.0411293092319811\\
4.74347545557602 	0.142231288373676\\
4.70388013720621 	-1.6464189497929\\
6.85043406286779 	-1.47956882288106\\
4.63506302310789 	-0.423414931299394\\
1.98719453542401 	-0.690913812259954\\
3.52051349477152 	-0.840601113141492\\
5.16242188366781 	-0.875458030119275\\
1.18499604445788 	-1.01186185663716\\
2.73634135547334 	-0.474227806442289\\
3.36693628503637 	-1.06421607914273\\
3.99465736914074 	-0.457605289274181\\
4.02828409269908 	-1.3426690863958\\
     };
        \addplot[only marks, color=greenalt] table[row sep=\\]{
        X Y\\
7.65913889497533 	0.119647811953381\\
9.24361764854281 	-0.324369679680127\\
6.30177064436666 	0.312741692919248\\
8.72750942033302 	-0.15948946303295\\
8.62552871072937 	0.153150125622065\\
8.34506907960978 	0.41840236739552\\
7.45940630414794 	1.14047168902221\\
10.7214632869453 	1.01335411500476\\
11.3234341684368 	1.46246483736245\\
6.6633112490754 	0.404922019201917\\
8.2684725190689 	1.13034477077416\\
8.2628303472134 	0.836333189343626\\
7.1004021681526 	1.03588126780285\\
9.56848654802772 	0.852064575885281\\
5.89563740229609 	-0.943270478985899\\
9.27470281784 	-0.282009953043022\\
9.41219517719067 	1.10436853706712\\
10.0004965861929 	0.432301809595737\\
8.27165420734646 	0.582593879353861\\
7.7365746487267 	1.26357030708004\\
8.01942150619114 	-0.0257985382029169\\
6.30430777512646 	0.0863941904262832\\
6.70315263759555 	-0.160441826013524\\
6.13343487405683 	0.916853174747333\\
9.02998081782813 	-0.483789911088027\\
7.67782407195951 	0.408288992964131\\
11.690861336369 	1.58996245682301\\
10.3334296300876 	0.83843842139794\\
6.20998795979553 	0.463541259949477\\
7.00364979453931 	0.903949498083473\\
11.5948581686107 	1.04758434390993\\
8.02302320421872 	0.428679751924055\\
9.06991412248317 	1.06593954019143\\
8.00756466978621 	1.30026932077329\\
9.8319209821211 	-0.0513565087957297\\
6.33467672708866 	0.542246829394773\\
9.64964861762538 	1.42963041012948\\
6.65757084747048 	-0.545754887136207\\
9.50400602736925 	0.409764257332989\\
8.07320211560372 	1.01607391579885\\
9.42676890718882 	0.975694687621875\\
7.82543661209891 	1.03642509564321\\
4.58860450828133 	0.289456349148762\\
3.62900115331203 	-0.24803283094032\\
8.92994310010732 	0.624226441773346\\
7.81875259160303 	0.321668668088018\\
5.89439568889025 	0.291463539799202\\
6.06367086300181 	0.252816992542708\\
8.60400248894434 	0.562715911715341\\
9.90526014444892 	1.38900102014854\\
7.9865840136475 	1.46963648394158\\
5.96945922574649 	-0.291418527880588\\
8.13506427014213 	0.293821517105055\\
8.80834308077157 	-0.0601010219442714\\
8.46062198793407 	1.1164748497904\\
8.3067205607234 	0.231350391046527\\
8.69798055622194 	0.415902539276036\\
6.5773729877227 	-0.00518266352904169\\
6.33684009852444 	0.0673087418354074\\
7.97910664739059 	0.57100922046291\\
8.38772171547729 	1.31643453515875\\
10.0609516768719 	1.85325598274487\\
5.62284140696902 	1.04156221635267\\
6.79943988117183 	0.34841767126168\\
6.9125300242126 	-0.330601903703326\\
7.46097388870054 	0.838912086527991\\
10.3224424113023 	-0.464538789771966\\
8.08786127311236 	0.654412327604428\\
10.8901639832269 	-0.368640174772179\\
8.33129433867914 	1.28558795135052\\
6.6682598301454 	1.36681599203068\\
9.78310031512761 	0.54615563667784\\
6.4689610171696 	0.175431255804895\\
12.2894876958393 	0.20732713201959\\
9.12027459775253 	1.74771198502513\\
10.5392203146212 	1.59835259332269\\
6.755050728484 	0.387171845218944\\
9.89053413735266 	-0.00152696817534323\\
8.41355141814627 	0.0351584367458867\\
8.3230455038379 	-0.595016178577037\\
8.9847238372319 	-0.0857925164342558\\
7.47352932674814 	0.924383166649717\\
5.62196042633697 	0.00283691548670029\\
8.77254483452356 	0.502829147674962\\
5.71813241319802 	0.206483168036558\\
8.18082867380565 	0.577027765704576\\
6.51002562991175 	0.856934890007073\\
8.03869197614856 	0.842292920719149\\
6.11471139786145 	0.966241978149288\\
8.39343873493165 	0.135672828764812\\
12.3143289158408 	1.53395287952325\\
9.9172678936236 	0.851259626161278\\
8.13857105259092 	0.670148404626568\\
8.58538768473008 	-0.198542078833184\\
7.86969170825909 	0.962477231475002\\
8.25800183630195 	0.694398966361671\\
7.35886967874128 	-0.292340355985333\\
7.263890691619 	1.33327780773623\\
11.6280960315484 	1.97368746607201\\
5.56001234669126 	-0.095008144684429\\
7.93495366264746 	0.777611399011765\\
9.05597619473267 	0.575518738533552\\
9.16374189938876 	1.18251780454195\\
10.4124510818549 	0.759193561910678\\
8.63036486566506 	0.758736678210859\\
9.83286068406799 	-0.709755658509859\\
9.55322147188098 	0.458477859339759\\
6.44552803562679 	0.470120253921498\\
8.5376463930434 	1.23264391542286\\
7.56134644584669 	-1.36009214372989\\
6.83901889282344 	0.0550282790052875\\
2.37701529925618 	0.185788202656193\\
8.67196532657235 	0.519013462601043\\
11.3035291158674 	0.672122371834183\\
6.58554619957717 	0.135249200037385\\
6.51542082691209 	-0.0601851985848321\\
8.33677519038951 	1.59371240103095\\
8.66299464075062 	1.54995464113938\\
9.83020911302949 	1.46346116016517\\
7.78708404534719 	0.990888445645224\\
8.50224623437508 	1.29125708627779\\
7.39841094620579 	1.5962530043629\\
8.77787899267665 	0.926293522597014\\
7.93743149177777 	0.218344624008958\\
10.4405471942654 	-0.0359742720330703\\
9.08785926346217 	1.42358371101829\\
10.0600583974226 	0.11578922561638\\
9.17223699691518 	1.11018827738655\\
8.82621490220095 	1.59616614865532\\
6.09714850188394 	0.79789239618292\\
6.91261264593271 	-0.624413426977089\\
7.91030894740512 	-0.13908221331218\\
8.35863434546316 	1.20542825359543\\
9.08856277371281 	1.22843028368968\\
9.49050742876334 	0.975717459146075\\
8.77258880212875 	0.621219536752735\\
8.38912837126808 	1.74778522120188\\
7.72585448017529 	0.312449322860457\\
7.80235460905776 	0.368963549854235\\
7.70037365202962 	0.237329435558116\\
12.3619551197422 	1.94501628358906\\
5.30694695343627 	0.458714338899979\\
7.3786820959745 	1.40768394338694\\
9.63766719738273 	1.366197308684\\
7.9007973332887 	0.260648112912961\\
9.22402193411761 	1.32442752438266\\
9.97602887338015 	0.113964781947934\\
5.90643753091212 	0.241233092575089\\
8.50520033610042 	0.924347916547323\\
10.6348930861331 	1.78047664499775\\
10.7298906212503 	0.790306427986175\\
8.71600677341448 	0.800963364132073\\
9.76805545319894 	1.0500842557383\\
7.44616808530719 	-0.295480284376531\\
7.63341171196389 	1.19724026436774\\
6.23936368735247 	-0.0195736556978068\\
11.5030431387781 	0.315716245241641\\
8.52503246808937 	0.92728373425816\\
6.4283101672243 	0.467894951130081\\
7.88569096924911 	0.369647184041593\\
    };
        \end{axis}
    \end{tikzpicture}
    \caption{\label{fig: Clust} \slshape\small Comparison of clustering potential between $\rho$-Laplacian spectral distance and DoD for 480 subsamples of size 50. The representation on the Euclidean plane is obtained by applying the classical Multidimensional Scaling algorithm to the distance matrices.} %
\end{figure}

%% file: Inference.tex
\newcommand{\DD}{\Delta\!\!\!\!\!\Delta}
\section{Statistical inference on mean spectra}
\label{sec: Inf}

\subsection{Sampling model and $\rho$-spectral estimator}
Given a random sample of mm-spaces of sample size $n$ \textemdash think of a collection of graphs representations as in Fig.~\ref{fig: Harmonics}, one can compute the sample of eigenvalues $\{\Lambda_i^\rho\}_{i=1}^n$. Let us first explain how we define a random sample of mm-spaces. Consider Sturm's space $(\mathbb{X}, \DD)$, the metric space of all metric measure spaces up to isomorphism \cite{sturm2012space}, and $\mathcal{P}(\mathbb{X})$ the set of all Borel probability measures on $\mathbb{X}$. For a fixed $\prob \in \mathcal{P}(\mathbb{X})$, we consider our dataset to be an i.i.d.\ sample of size $n$ generated by $\prob$. Such a setting has also been used to define and study barycenters of mm-spaces; see \cite{gouic2019fast} for example.
Further, each random, finite mm-space is a random variable taking its value in $S^K_+ \times \reals^K$. It is thus a distance matrix along with a weight vector. Denote by $F_{d_{j\ell}\vert \mu}$ the cumulative distribution function of the distance between the points $X_j$ and $X_\ell$ conditionally on the the weight vector, $\mu$.

Based on the empirical probability measure $\prob_n^\rho$ from the sample $\{\Lambda_i^\rho\}_{i=1}^n$, one defines an estimator of the true spectrum $\lambda^{\rho}$ as 
\begin{equation}
\label{eq: LamEstForm}
(\widehat{\lambda^{\rho}})_k = \frac{1}{n}\sum_{i=1}^n \Lambda_{i,k}^\rho, \quad k \in \llbracket K\rrbracket. 
\end{equation}
One can then compare two such means to infer on differences between mm-spaces. This can be based on asymptotic considerations $(n\to \infty)$ as well as on finite sample concentration bounds, which will be detailed  in the following section.

\subsection{Asymptotic distribution result and finite sample result.}
Before stating the assumption and proving the main result of this section, we recall an important lemma that we will use in the sequel. It is  
 a direct consequence of Gershgorin's circle theorem (\cite{gershgorin1931uber}). %

\begin{lemma}
\label{lem: BoundEigGraph}
The largest eigenvalue of a weighted graph's Laplacian with $K$ nodes, $\lambda_K$, is such that 
\[
\lambda_K \le  2\max_{j} \operatorname{Deg}_{j,j}.
\]
\end{lemma}

In order to state our main theorem, we require the space of functions 
\[
\ell_d^\infty([0,D]) := \left \{ z : [0,D] \to \reals^d\ \Big\vert \sup_t\ \lVert z(t) \rVert_\infty < \infty \right \},
\]
which is the product space of $\ell^\infty([0,D])$. 

\begin{theorem} 
\label{thm: LimThm}
Consider the estimator proposed in Eq.~\eqref{eq: LamEstForm}. Assume that the total mass of each metric measure space is bounded above and that all the distances are smaller than $D$. Then, for any $\rho \in [0,D]$,
\[
\sqrt{n} \left( \widehat{\lambda^{\rho}} - \lambda^{\rho}\right) \dto \Gproc(\rho),
\]
where  $\Gproc(\rho)$ is a centred Gaussian process depending solely on $\rho$ and where weak convergence is w.r.t.\ $\ell_d^\infty([0,D])$.
\end{theorem}

\bgroup
\begin{rmk}[Smooth functionals of $\Gproc(\rho)$]
Theorem~\ref{thm: LimThm} extends to suitable differentiable functions of the $\rho$-spectrum, as the limit law process can then be determined by the delta method \cite[Chapter~20]{vaart_1998}. For instance, in the case of metric spaces endowed with the uniform measures, the number of spanning trees of the auxiliary graph is such a function. Indeed, Proposition~\ref{sec: spanTrees} states that, for sufficiently large $\rho$, the number of spanning trees is determined by the product of the non-null eigenvalues.
\end{rmk}
\egroup

\begin{proof}[Proof of Theorem~\ref{thm: LimThm}]

We will prove that each coordinate $\sqrt{n} ( \widehat{\lambda^{\rho}_j} - \lambda^{\rho}_j)$, $j \in \llbracket K \rrbracket$, of the above process converges weakly in $\ell^\infty([0,D])$ to a separable limit and conclude that the joint process converges to a separable limit as in \cite[Theorem~1.4.8]{van1996weak}.

Before addressing the convergence of each coordinate, let us establish the joint convergence of $\sqrt{n} ( \widehat{\lambda^{\rho}}  - \lambda^{\rho}) =: \mathbf{\Lambda}_n^\rho$ for a finite collection $\rho_1, \ldots, \rho_m$. 
Recall the linear form of the estimator given in Eq.~\eqref{eq: LamEstForm} and consider a fixed set $\rho_1, \ldots, \rho_m$. As the observations are independent and the spectra bounded by Lemma~\ref{lem: BoundEigGraph}, the central limit theorem in $\reals^{K\times m}$ provides the sought joint convergence. 

Weak convergence to a Gaussian process of each coordinate will follow from \cite[Theorem~1.5.4]{van1996weak}. 
To this end, we need to establish convergence of the marginals of each coordinate and asymptotic tightness. The first requirement follows from the CLT similarly to the joint convergence of  $\{\mathbf{\Lambda}_n^\rho\}_{\rho \in \{\rho_1, \ldots, \rho_m\}}$ established just above.
It thus remains to prove the asymptotic tightness part. From \cite[Theorem~1.5.7]{van1996weak}, we know that the latter is implied by total boundedness of $([0,D], \dist)$ and $\dist$-equicontinuity, i.e.,
 for all $\eps, \eta>0$, there exists $\delta>0$ such that  
\begin{equation}
\label{eq: AEC}
\limsup_{n \to \infty} \prob^* \left( \sup_{\dist(s,t)<\delta}  \left \lvert \mathbf{\Lambda}_{n,j}^s -\mathbf{\Lambda}_{n,j}^t \right\rvert > \eps \right) < \eta,
\end{equation}
for a semi-metric $\dist$, with $\prob^*$ denoting outer probability. Instead of working with the maximum difference between the sorted eigenvalues, we will show that equicontinuity holds irrespective of the eigenvalue difference.
As total boundedness and equicontinuity are equivalent to showing that the class of function is $\prob$- Donsker \cite[Corollary~2.3.12]{van1996weak}, we derive the latter property from a bracketing argument that we now develop.
 
By choosing different $\rho$ parameters called $s$ and $t$ in the sequel, with $s<t$, one can bracket the class of function $\mathcal{F}:=(\lambda_k^\rho)_{\rho \in [0,D]}$. The dependence of the eigenvalue on the distance matrix and the probability measure of the metric-measure space are silenced to ease the reading. The fact that these brackets for the eigenvalues fulfil the definition of brackets, for any element of $S^K_+ \times \reals^K$, is a direct consequence of Cauchy's interlacing theorem and the fact that the smallest eigenvalue of a graph Laplacian is zero. 
Further note that this class of functions is pointwise measurable, i.e., the countable subset 
\[
\mathcal{G}:=(\lambda_j^\rho)_{\rho \in [0,D] \cap \QQ}
\] 
 is such that for each $f$ in $\mathcal{F}$ there exists a sequence $g_m$ in $\mathcal{G}$ with $g_m(x) \to f(x)$ for every $x \in S^K_+ \times \reals^K$. This is a mere consequence of the fact that there exists a sequence of rational numbers converging to any real number. 
Further, using boundedness and constructing sequences in $[0,D]$ converging to the elements of a maximising sequence, it  can be shown that 
\[
\sup_{\rho \in [0,D]} \lvert \mathbf{\Lambda}_{n,j}^\rho \rvert = \sup_{\rho \in [0,D]\cap\QQ } \lvert \mathbf{\Lambda}_{n,j}^\rho \rvert.
\]
As the supremum on the right hand side is measurable, the outer expectations in the latter case reduce to classical ones. 
To avoid notational burden, we will remove the rational condition of $\rho$ in the sequel.
It further holds that
\begin{align*}
B_{s,t}:=\expec \left(  \lvert \lambda_j^s - \lambda_j^t \rvert  \right) & \le \expec\left(  \left\lVert \sum_{e \in E} \mu_{in(e)}\mu_{out(e)} \1_{\{s<d_e\le t \}} b_e b_e^\top \right \rVert_{op}  \right) \\
&  \le  \expec\left(  \left(2 \max_j \mu_{j} \sum_{j\sim l} \mu_{\ell} \1_{\{s<d_{j,\ell}\le t \}} \right )  \right), 
\end{align*}
where the first inequality follows from the fact that, for two $n \times n$ matrices $A,B$,
\[
\lvert \lambda_i(A+B) - \lambda_i(A) \rvert \le \lVert B\rVert_{op} := \max(\lvert \lambda_1(B) \rvert, \lvert \lambda_n(B) \rvert ).
\]
The second inequality holds by Lemma~\ref{lem: BoundEigGraph}. %
We then have, using the tower property of the expectation,
 \begin{align*}
B_{s,t}& \le  \expec\left( \max_j \mu_{j} \sum_{j\sim l} \mu_{\ell} \prob \left (s<d_{j,\ell}\le t  \vert \mathbf{\mu} \right )  \right) .
\end{align*}
From this inequality, a set of points can be chosen and $\eps$-brackets in $L_1(\prob)$ be constructed. To ensure that each of the $K(K-1)/2$ distribution of distances be well covered, it is sufficient to take the union of the points that would give $\eps$-brackets for the distribution of distance between one single pair $j,\ell$. By assumption, the total mass of each mm-space is bounded. Let us denote this bound by $\mathfrak{m}$. We can further notice that, up to rescaling by $2  \mathfrak{m}^2$ and exploiting homogeneity of the norms, one can control the $L_2$ bracketing number from the $L_1$ one. We thus have that 
\begin{equation}
\label{eq: brackNum}
\normal_{[]}(\eps 2\mathfrak{m}^2 ,\fclass, L_2(\prob)) \le K (K-1) \eps^{-2}.
\end{equation}
The bracketing integral is thus finite and it follows from \cite[Theorem~19.5]{vaart_1998} that the class of functions is $\prob$-Donsker, which implies tightness.%
\end{proof}

\subsection{Two sample test}

In this section, we address the question of comparing statistically the mean $\rho$-spectra, potentially truncated, of two independent samples of sizes $n$ and $m$, respectively. 
That is, given two samples of mm-spaces with associated expected $\rho$-spectra ${\Lambda}_1^\rho$ and ${\Lambda}_2^\rho$, we consider the hypothesis testing problem
\begin{equation}
\label{eq: NullHyp}
 H_0 : {\Lambda}_1^\rho = {\Lambda}_2^\rho, \forall \rho \qquad \text{vs.} \qquad H_1: \exists \rho : {\Lambda}_1^\rho \neq {\Lambda}_2^\rho. \tag{*}  
\end{equation}
We suggest to recourse on the following test statistic  
\[
T_{n,m} := \sup_\rho  \big \lVert \widehat{\lambda_{n,1}^\rho}- \widehat{\lambda_{m,2}^\rho} \big\rVert_\infty,
\]
where $\widehat{\lambda_{n,1}^\rho}$ denotes the estimator of ${\Lambda}_1^\rho$ obtained as in Eq.~\eqref{eq: LamEstForm}.
and thus consider the test 
\[
(c)\qquad \phi := \1(T_{n,m} > c_\alpha ), 
\]
for any prescribed significance level $\alpha \in [0,1]$. To apply the test, the distribution of $T_{n,m}$ is required. From this, the critical value $c_\alpha$ can be obtained as a quantile. We propose two ways to achieve this that we develop in the next subsections.
\bgroup
\begin{rmk}[mm-spaces of different size]
In the case where the mm-spaces do not have the same size, we consider truncated spectra as is done for ShapeDNA in \cite{reuter2006laplace}. That means that we only keep the $K'$ smallest eigenvalues. This choice is made for the proposed approach to be a coherent generalisation of classical data analysis on manifolds. It also makes sense as the Euclidean distance between the truncated spectra involves summing squared distances between individual eigenvalues that have a comparable meaning.  Recall that the distance between $\lambda_2^\rho$ and $\tilde{\lambda}_2^\rho$, say, expresses a difference of connectivities of the underlying auxiliary graphs, for instance. 

Another approach, that we do not pursue here, could be to identify each $\rho$-spectrum with its spectral measure $\sum_{j\in \llbracket K \rrbracket} \delta_{\lambda_j^\rho}/K$. For each group, one could compute a Wasserstein barycenter and then compare the two groups using the Wasserstein distance between their respective Wasserstein barycenters. In the case where each mm-space of the sample has the same number of nodes, this latter approach coincide with the one above.   \end{rmk}
\egroup
\subsubsection{Concentration inequality}
\bgroup
A first way to carry out the test in $(c)$ is to recourse to concentration bounds. They have the advantage of being valid for all $n,m$. Unfortunately, they will depend on a bound on the mass and a universal constant. Therefore, the subsequent bound is only practically useful when these constants can be made explicit. 

\begin{proposition}
Consider two independent samples of finite mm-spaces as above with finite sample sizes $n$ and $m$ and assume that the total mass of each mm-space is bounded above by $\mathfrak{m}$. Then, it holds that
\[
\prob ( T_{n,m} > t) \le \sum_{j=2}^K \sum_{\xi \in \{n,m\}}  \exp \left(- \frac{ (t/2- \expec Z_j)^2 }{2 \nu_{j,\xi} +  2  \mathfrak{m}^2 t  /3}\right) ,  
\]
with 
\begin{align*}
\nu_{j,\xi} &\le 16\mathfrak{m}^4 \kappa \sqrt{ \frac{2}{\xi}\log\left( \max(K/ \mathfrak{m} , e^2 \mathfrak{m})  \right)} + 4 \xi\mathfrak{m}^4, \\
\expec Z_{j,\xi} &\le \kappa 4 \mathfrak{m}^2   \sqrt{ \frac{2}{\xi}\log\left( \max(K/ \mathfrak{m} , e^2 \mathfrak{m})  \right)}
\end{align*}
and $\kappa$ is the universal constant in Ossiander's bracketing theorem \cite[Proposition~6.7]{talagrand1994sharper}.
\end{proposition}
\begin{proof}
First notice that, under the null hypothesis in~\eqref{eq: NullHyp}, 
\[
\prob ( T_{n,m} > t) \le  \prob \left( \sup_\rho \lVert \mathbf{\Lambda}_n^\rho \rVert_\infty > \frac{t\sqrt{n}}{2} \right) +  \left( \sup_\rho \lVert \mathbf{\Lambda}_m^\rho \rVert_\infty  > \frac{t\sqrt{m}}{2} \right).  
\]
We will use the union bound and Bousquet's inequality (\cite{bousquet2003concentration}) to control each term in the right hand side of the above display. 
The union bound seems unavoidable as controlling each eigenvalue is far from trivial.  Setting $Z_j:= \sup_\rho\lvert \widehat{\lambda^{\rho}_{1,j}} - \lambda^{\rho}_j\rvert$, in this case, the Bousquet inequality  reads 
\[
\prob( Z_j > \expec Z_j + t ) \le \exp\left( - \frac{t^2}{2 \nu_{j,n} + 2t U /3} \right),
\]
where $U$ is the bound on the envelope function and, here,  
\[
\nu_{j,n} = 2U \expec Z_j + \sum_{i =1}^n  \sup_\rho \expec \left\lvert \widehat{\lambda^{\rho}_{1,j}} - \lambda^{\rho}_j\right\rvert^2 \le 2U \expec Z_j+ nU^2.
\] 
The main element remaining is the control of $\expec Z_j$.  To this end, we will use Ossiander's bracketing theorem. 
The latter reads, for a general function class $\fclass$,
\[
\expec \sup_{f \in \fclass} \left \lvert \sum_{i=1}^n f(X_i) - n \expec f \right \rvert  \le \kappa \sqrt{n} \int_0^\infty \sqrt{\log \normal_{[]}(\eps  ,\fclass, L_2(\prob))} \ \diff \eps,
\]
 for some universal constant $\kappa$.
From the latter we have by Eq.~\eqref{eq: brackNum} that, for all $j \in \llbracket K \rrbracket$, 
\begin{align*}
n \expec Z_j 
&\le \kappa \sqrt{n} \int_0^{2\mathfrak{m}^2} \sqrt{\log\left( 4 \mathfrak{m}^2 K(K-1) \eps^{-2}\right)} \ \diff \eps \\
&\le \kappa \sqrt{2n} \int_0^{2\mathfrak{m}^2} \sqrt{\log\left( 2 \mathfrak{m} K \eps^{-1}\right)} \ \diff \eps \\
& \le \kappa 4 \mathfrak{m}^2 \sqrt{2n}  \sqrt{ \log\left( \max(K/ \mathfrak{m} , e^2 \mathfrak{m})  \right)} , 
\end{align*}
where we used the fact that if $\log (C/c) \geq 2$, 
\[
\int_0^c \sqrt{\log\left(\frac{C}{x}\right)} \diff x \le 2c \sqrt{\log\left(\frac{C}{c}\right)}.
\]
It follows that 
\[
\nu_{j,n} \le 16\mathfrak{m}^4 \kappa \sqrt{ \frac{2}{n}\log\left( \max(K/ \mathfrak{m} , e^2 \mathfrak{m})  \right)} + 4n\mathfrak{m}^4.
\]
By the union bound, it further holds 
\begin{align*}
 \prob \left( \sup_\rho \lVert \mathbf{\Lambda}_n^\rho \rVert_\infty > \frac{t\sqrt{n}}{2} \right) 
 &\le \sum_{j=2}^K  \prob \left( Z_j > \frac{t}{2}\right)\\
 & \le\sum_{j=2}^K \exp \left(- \frac{ (t/2- \expec Z_j)^2 }{2 \nu_{j,n} + 2(t/2- \expec Z_j) 2 \mathfrak{m}^2  /3}\right).
\end{align*}
The claim follows.
\end{proof}

\begin{rmk}[Room for improvement]
Looking at the proof, there are various places where additional knowledge or regularity assumptions would guarantee sharper inequalities. 
Additional assumptions for the distances would help to improve the factor $K(K-1)$ and assumptions on $\max_{j \in \llbracket K \rrbracket}\mu_j$ would reduce the bracketing number.
Finally note that a refined understanding of the Laplacian spectrum could also greatly help, in particular to avoid using the union bound. %
\end{rmk}
\egroup
\subsubsection{Bootstrap}
Note that, under the null hypothesis in~\eqref{eq: NullHyp}, it holds that 
\begin{equation}
\label{eq: limProcTest}
T_{n,m} \dto \sup_\rho \big\lVert \Gproc(\rho) - \Gproc'(\rho)\big \rVert_\infty, \qquad n \to \infty, \tag{**}
\end{equation}
by the continuous mapping theorem, where $\Gproc(\rho)$ and $\Gproc'(\rho)$ are two independent copies of the limiting process from Theorem~\ref{thm: LimThm}.

As it is difficult to derive the distribution \eqref{eq: limProcTest} under the null hypothesis explicitely, we suggest to recourse to a bootstrap procedure. 
First, consider $\{X_i^\star\}_{i=1}^n$ a bootstrap sample from $\prob_n$ and its associated empirical measure $ \prob_n^\star$. The bootstrap process is defined as 
\[
\big\{\Gproc_n^\star f\big\}_{f \in \fclass} := \left\{\sqrt{n}\big( \prob_n^\star f- \prob_n f\big)\right\}_{f \in \fclass}. 
\]
The procedure then goes as follows  \cite[Chapter~16]{efron1994introduction}.

\begin{enumerate}
\item Generate $B$ bootstrap samples for the two samples $\Gproc_{n,1}^{\star,\mathfrak{b}}$ and  $\Gproc_{m,2}^{\star,\mathfrak{b}}$ with $\mathfrak{b} \in \llbracket B \rrbracket$. 
\item Compute the $B$ realisations of 
\[
 \theta_\mathfrak{b}:= \sup_\rho \lVert \Gproc_{n,1}^{\star,\mathfrak{b}}(\rho) - \Gproc_{m,2}^{\star,\mathfrak{b}}(\rho) \rVert_\infty.
\]
\item Approximate the p-value by 
\[
\hat{p}_{n,obs}:= \frac{1}{B} \sum_{\mathfrak{b} = 1}^B \1(T_{n,m,obs} <  \theta_\mathfrak{b} ),
\]
where $T_{n,m,obs}$ is the observed value of the test statistic $T_{n,m}$.
\end{enumerate}
Let $p_{n,obs}$ denote the p-value that one would compute, should one know the distribution of $\Gproc_{n,1}(\rho)$ \textemdash and therefore $\sup_\rho \lVert \Gproc_n(\rho) - \Gproc_n'(\rho) \rVert_\infty$  under the null hypothesis. 
\begin{proposition}
The bootstrap procedure is consistent for the testing problem as
\[
\lvert \hat{p}_{n,obs} -{p}_{n,obs} \vert \overset{\prob}{\to} 0, \quad \text{as } n,B \to \infty.
\]
\end{proposition}
\begin{proof}
Note that from the convergence provided by Theorem~\ref{thm: LimThm}, the considered class of functions is $\prob$-Donsker. Theorem~2.4 in \cite{gine1990bootstrapping}
states that\footnote{The statement is slightly reworded in comparison to the original version and appears in \cite{chazal2013bootstrap}.} a function class
$\fclass$ is $\prob$-Donkser if and only if $\Gproc_n^\star$ converges in distribution to $\Gproc$ in $\ell_\infty(\fclass)$. This theorem holding for each component can be used to conclude that the convergence holds for the entire distribution using Theorem~1.4.8 in \cite{van1996weak}. The claim then follows from the law of large numbers.
\end{proof}

\bgroup
\subsection{Example}
\label{sec: exInference}
Recall the example in Section~\ref{sec: firstExamples}. For the largest eigenvalue and the second smallest or Fiedler eigenvalue we represent the mean over 11 replications for the camels and 10 replications for the lions along with pointwise  95\%-confidence bands based on student quantiles. The rationale behind the choice of these two eigenvalues is that the Fiedler eigenvalue is a measure of connectivity, while the largest eigenvalue is driven by (a transformation of) the weighted maximum local distribution of distances. 

\begin{figure}[h!]
\label{fig: Inf}
\begin{center}
\includegraphics[width=0.45\textwidth]{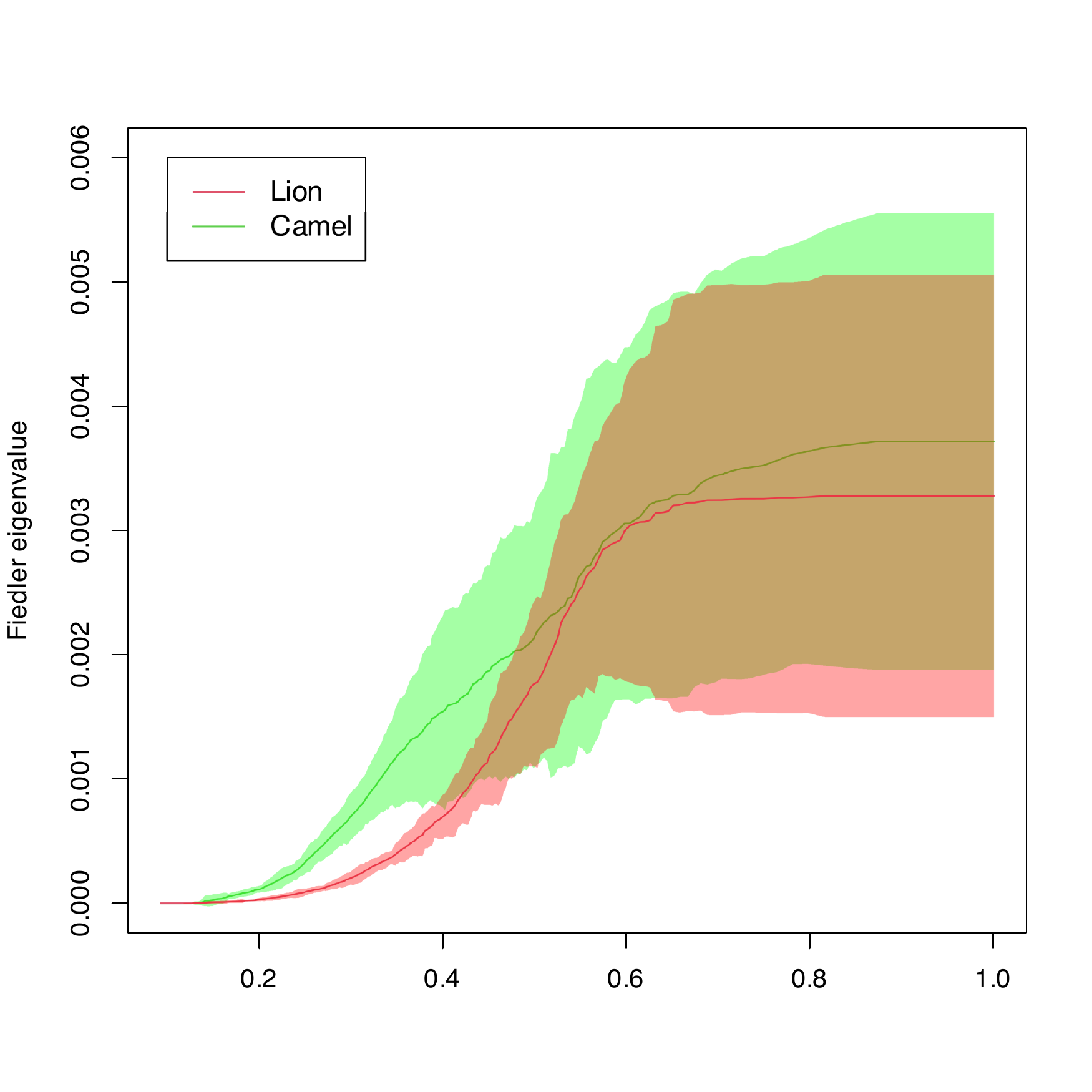}
\includegraphics[width=0.45\textwidth]{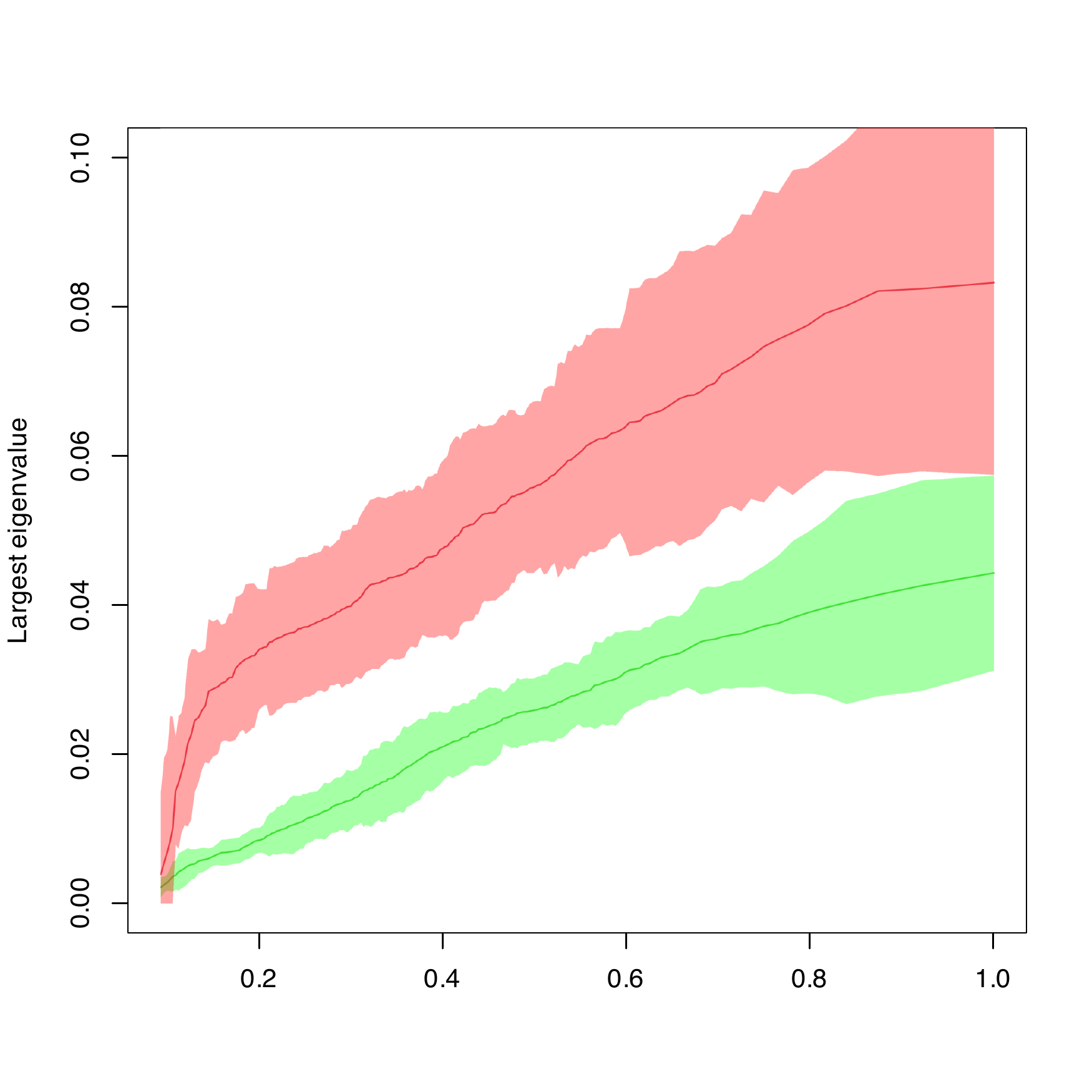} \\
\vspace{-0.8cm}
\hfill $\rho $ \hfill $\rho \qquad \qquad$
\end{center}
\caption{ Left) Pointwise means and confidence bands for the Fiedler eigenvalue as a function of $\rho$. The mean is given by the continuous line while the confidence bands are represented by the shaded area of the corresponding colour. Right) Pointwise means and confidence bands for the largest eigenvalue as a function of $\rho$. }
\end{figure}

We clearly find from Fig.~\ref{fig: Inf} that the connectivities of the auxiliary graphs are different for small values of $\rho$, while the largest eigenvalues have a completely different behaviour at all $\rho$-scales. 
\egroup